\documentclass[11pt, a4paper]{amsart}
\usepackage{amsfonts,amssymb,amsmath, amsthm}
\usepackage{mathrsfs}
\usepackage{lmodern}
\usepackage{qsymbols}
\usepackage{latexsym}
\usepackage[noadjust]{cite}
\usepackage{pdfsync}
\usepackage{bm}
\usepackage[cmtip,all]{xy}
\usepackage{enumitem}

\newtheorem{thm}{Theorem}[section]
\newtheorem{lem}[thm]{Lemma}
\newtheorem{prop}[thm]{Proposition}
\newtheorem{cor}[thm]{Corollary}

\theoremstyle{definition} 

\newtheorem{rem}[thm]{Remark}

\numberwithin{equation}{section}
\usepackage[plainpages=false,pdfpagelabels,
citecolor=red]{hyperref}
\usepackage{xcolor}
\hypersetup{
 colorlinks,
 linkcolor={cyan!90!black},
 citecolor={magenta},
 urlcolor={green!40!black}
} 
\newcommand{\R}{\mathbb{R}}
\newcommand{\IC}{\mathbb{C}}

\newcommand{\cH}{\mathcal{H}}

\newcommand{\cL}{\mathcal{L}}

\newcommand{\C}{\operatorname{C}} 
\newcommand{\W}{\operatorname{W}}
\newcommand{\reu}{{\mathbb{R}^{1+n}_+}}
\newcommand{\ree}{{\mathbb{R}^{1+n}}}

\renewcommand{\P}{D} 
\newcommand{\M}{B} 

\renewcommand{\div}{\operatorname{div}}

\newcommand{\gradA}{{\nabla_{\! A}}} 
\newcommand{\dnuA}{\partial_{\nu_A}} 
\newcommand{\pe}{\perp}
\newcommand{\pa}{\parallel}

\newcommand{\NT}{\widetilde{N}_*} 
\newcommand{\NTq}{\widetilde{N}_{*,q}} 
\newcommand{\NTone}{\widetilde{N}_{*,1}} 
\newcommand{\HL}{\mathcal{M}}

\newcommand{\e}{\mathrm{e}} 
\let\ii\i
\renewcommand{\i}{\mathrm{i}} 
\newcommand{\eps}{\varepsilon} 
\renewcommand\Re{\operatorname{Re}}

\newcommand{\Lop}{\mathcal{L}} 
\newcommand{\Le}{\mathcal{L}}
\DeclareMathOperator{\supp}{supp} 
\DeclareMathOperator{\ran}{\mathsf{R}} 
\DeclareMathOperator{\dom}{\mathsf{D}} 

\newcommand{\dual}[2]{\langle #1,#2 \rangle}

\newcommand{\clos}[1]{\overline{#1}}

\newcommand{\tdd}[2]{\tfrac{\partial #1}{\partial #2}}

\newcommand{\wt}{\widetilde}
\newcommand{\con}[1]{\overline{#1}}
\newcommand{\ta}{{\scriptscriptstyle \parallel}}
\newcommand{\no}{{\scriptscriptstyle\perp}}
\newcommand{\pd}{\partial}

\newcommand{\mD}{{\mathcal D}}
\newcommand{\mS}{{\mathcal S}}

\def\Xint#1{\mathchoice
{\XXint\displaystyle\textstyle{#1}}%
{\XXint\textstyle\scriptstyle{#1}}%
{\XXint\scriptstyle\scriptscriptstyle{#1}}%
{\XXint\scriptscriptstyle%
\scriptscriptstyle{#1}}%
\!\int}
\def\XXint#1#2#3{{\setbox0=\hbox{$#1{#2#3}{%
\int}$ }
\vcenter{\hbox{$#2#3$ }}\kern-.6\wd0}}
\def\barint{\,\Xint -} 
\def\bariint{\barint_{} \kern-.4em \barint}
\def\bariiint{\bariint_{} \kern-.4em \barint}
\renewcommand{\iint}{\int_{}\kern-.34em \int} 
\renewcommand{\iiint}{\iint_{}\kern-.34em \int} 

\title[On uniqueness  for Dirichlet  problems]{On uniqueness results for Dirichlet problems \\ of elliptic systems without \\ 
DeGiorgi-Nash-Moser regularity}
\author{Pascal Auscher}
\author{Moritz Egert}
\address{Laboratoire de Math\'{e}matiques d'Orsay, Univ. Paris-Sud, CNRS, Universit\'{e} Paris-Saclay, 91405 Orsay, France; and Laboratoire Ami\'enois de Math\'ematique  Fondamentale et Appliqu\'ee, 
CNRS-UMR 7352, Universit\'e de Picardie-Jules Verne, 80039 Amiens France}
\email{pascal.auscher@math.u-psud.fr}
\address{Laboratoire de Math\'{e}matiques d'Orsay, Univ. Paris-Sud, CNRS, Universit\'{e} Paris-Saclay, 91405 Orsay, France}
\email{moritz.egert@math.u-psud.fr}
\thanks{The authors were partially supported by the ANR project ``Harmonic Analysis at its Boundaries'', ANR-12-BS01-0013. This material is based upon work supported by the National Science Foundation under Grant No.\ DMS-1440140, while Auscher was in residence at the MSRI in Berkeley, California, during the Spring 2017 semester. Egert was supported by a public grant as part of the FMJH and also thanks the MSRI for hospitality.}
\date{\today}
\dedicatory{}
\begin{document}
\begin{abstract}
We study uniqueness of Dirichlet problems of second order divergence-form elliptic systems with transversally independent coefficients on the upper half-space in the absence of regularity of solutions. To this end, we develop a substitute for the fundamental solution used to invert elliptic operators on the whole space by means of a representation via abstract single layer potentials. We also show that such layer potentials are uniquely determined. 
\end{abstract}
\maketitle
\section{Introduction}
\label{Sec: Introduction}
Consider the elliptic system of $m$  equations in $n+1$ dimensions, $n\ge 1$, given by
\begin{equation}  \label{eq:divform}
  -\sum_{i,j=0}^n\sum_{\beta= 1}^m \pd_i\big( A_{i,j}^{\alpha, \beta}(x) \pd_j u^{\beta}(t,x)\big) =0,\qquad \alpha=1,\ldots, m, \qquad  t>0, x\in \R^n,
\end{equation}
where $\pd_0:= \tdd{}{t}$ and $\pd_i:= \tdd{}{x_{i}}$ if $i=1,\ldots,n$ with measurable coefficients $A$ that do not depend on the variable $t$ transversal to the boundary. Ellipticity will be described below but when $m=1$, the uniformly elliptic equations will be included in our considerations. For short, we shall write $\Le u=-\div A \nabla u=0$ instead of \eqref{eq:divform}.

Given $f\in L^p(\R^{n}; \IC^{m})$, following  \cite{Da}, the $L^p$ Dirichlet problem on the upper half-space  can be posed in the sense that one asks for a weak solution $u$ with a certain non-tangential maximal function controlled in $L^p$ and which converges to the boundary data $f$ almost everywhere in a non-tangential sense. 
When $f\in \dot W^{1,p}(\R^n; \IC^m)$, following  \cite{KP}, the Dirichlet problem with data $f$, also known as the regularity problem, can  be posed by asking for a maximal non-tangential control on $\nabla u$   and convergence of $u$ to $f$ at the boundary as before.  
Existence and uniqueness to these problems are usually obtained by different arguments. For an overview on the topic the reader can refer to \cite{Ke}.

Our first goal is to prove duality results of the following type under minimal assumptions: existence in one of the boundary value problems for $(\Le^*,p')$ implies uniqueness in the other problem for $(\Le, p)$ in some range of $p$, which depends on $\Le$, where $p'$ is the conjugate exponent to $p$. We shall also consider the case $p\le 1$ for the regularity problem, in which case the adjoint Dirichlet problem must be posed with data in $BMO$ or in a H\"older space.  

Similar uniqueness results, requiring ``dual'' information, appear  in  \cite{KP, AAAHK, HMaMo}, to cite just the most relevant to our situation. The arguments are  culminations of many earlier results on Laplace's equation and real symmetric equations in Lipschitz domains (\cite{Da, DaK, JK, V}).  In those works,  $t$-independence of the coefficients is not always assumed, but when restricted to this hypothesis, the so-called de Giorgi--Nash--Moser regularity properties of solutions (DGNM) are also used in a strong way to bring into play either harmonic measure techniques for real equations or representations and estimates with fundamental solutions for complex equations enjoying (DGNM). It seems that \cite{HMaMo} contains the most advanced results in this direction up to now. 

Here we want to dispense with the assumption (DGNM) and, of course, harmonic measure is not available. In a similar direction, \cite{AM} establishes existence-uniqueness relations between the $L^p$ regularity problem and a dual $L^{p'}$ Dirichlet problem which for $1<p<\infty$ is posed with a different, less classical interior control, namely the square function. Uniqueness  in this situation, however, does not suffice to conclude for uniqueness of the Dirichlet problem when posed with a non-tangential maximal control. On the contrary, when $p\le 1$, the results in \cite{AM} do apply and for clarity we shall put them into context in Section~\ref{sec:other}.

Our general strategy is to develop a substitute for the fundamental solution used to invert the elliptic operator $\Le$ on $\ree$. This is interesting its own right. Surprisingly, not using the fundamental solution and its kernel estimates  will make the arguments for uniqueness conceptually and technically simpler. It also allows us to reach minimal assumptions, even when assuming further (DGNM). Let us explain in formal terms the substitution idea. 

In the case of transversally independent coefficients, the fundamental solution $\Gamma(t,x, s,y)$ of $\Le$, constructed in \cite{HK} under (DGNM) and more recently without this assumption in \cite{B}, has time translation invariance, that is, it depends on $t-s$. Its restriction to fixed times $(t,0)$, $t\ne 0$, is called the single layer potential $\mS_{t}(x,y)$ at time $t$. Formally writing 
\begin{align*}
 (\Le^{-1}f)(t,x)= \iint_{\ree} \Gamma(t,x, s,y) f(s,y)\ ds \, dy= \iint_{\ree} \mS_{t-s}(x,y) f(s,y)\ dy \, ds,
\end{align*}
allows one to recover the fundamental solution by a convolution in time with $\mS_{t}(x,y)$. A difficulty is to give a meaning to the last term as a converging integral in order to obtain further estimates on $\Le^{-1}f$. However, one can interpret this formula at the level of operators by writing
\begin{equation}
\label{eq:L-1}
(\Le^{-1}f)(t,x)=   \int_{\R} (\mS_{t-s} f(s,\cdot))(x) \ ds,
\end{equation}
provided the   operator $\mS_{t}$ with  kernel  $\mS_{t}(x,y)$ has the expected boundedness properties. Indeed, \cite{R} shows  the remarkable fact that  $\mS_{t}$ is bounded from $L^2(\R^n; \IC^m)$ into $\dot W^{1,2}(\R^n;\IC^m)$  whether or not (DGNM) holds and that, when (DGNM) is assumed, its kernel agrees with (or can be used to define) $\Gamma(t,x,0,y)$. 

This suggests that knowledge on the operator $\mS_{t}$ alone is sufficient to recover $\Le^{-1}$.  This is what we shall prove and use, thereby giving a precise meaning  to the representation \eqref{eq:L-1}. We shall also prove that  knowledge of $\Le^{-1}$ alone uniquely determines the operator  $\mS_{t}$, which we decide to call the single layer operator (associated with $\Le$). 

Having \eqref{eq:L-1} at hand, more operator bounds of $\mS_{t}$ can be plugged in this formula to give further estimates on $\Le^{-1}f$. Under (DGNM), \cite{HMiMo} proves some bounds   by Calder\'on-Zygmund theory. But, following \cite{R}, we may also compute    $\mS_{t}$ (recall it is unique) using the connection between $\Le$ and a first order Dirac operator $DB$    proposed in \cite{AAMc}. Thus,   the operator bounds  proved in \cite{AusSta} become available. Such bounds, including the ones of \cite{HMiMo}, hold for a range of spaces determined by the coincidence of abstract Hardy spaces associated with $DB$ and  the corresponding concrete Hardy spaces associated with $D$.  At the heart of this treatment lies the $H^\infty$ functional calculus of $DB$ proved in \cite{AKMc} by a remarkable elaboration on the solution of the Kato problem for elliptic systems.  

The organisation of the article is as follows. First, we present our main results and the strategy to prove uniqueness (Section \ref{sec:strategy}). We next present proofs of our main results in the case $p=2$ because the arguments there do not require any use of the single layer operators and still contain the main ideas (Section \ref{sec:p=2}). Then, we state in what sense \eqref{eq:L-1} holds  (Section \ref{sec:layer}) and  move to $p\ne 2$ (Section \ref{sec:pnot2}).  
In  Section \ref{sec:other} we discuss the regularity problem with Hardy-Sobolev data versus the Dirichlet problem with $BMO$ or H\"older continuous data. We prove  \eqref{eq:L-1} in various ways (Section \ref{sec:prooflayer}).
Some technical lemmas are presented in the final Section \ref{sec:technical}.
\section{Setup, results and strategy of proofs}\label{sec:strategy}

\subsection{Notation and general assumptions}

We shall use the following notation for spaces. We denote by $C_{0}^\infty(\R^d)$  the space of compactly supported smooth complex-valued functions on $\R^d$. For $1< p <\infty$, the inhomogeneous Sobolev space on $ \R^d$ consists of those $f \in L^p(\R^d; \IC)$ for which $\nabla f$ is $p$-integrable. It contains $C_{0}^\infty(\R^d)$ as a dense subspace. The homogeneous Sobolev space $\dot W^{1,p}(\R^d)$ consists of all distributions on $\R^d$ for which $\nabla f$ is in $L^p(\R^d; \IC^{d})$. It is a Banach space when moding out the constants and it can be realised as the closure of $W^{1,p}(\R^d)$ modulo constants for the semi-norm  $\|\nabla f\|_{p}$. Its dual $\dot W^{-1,p'}(\R^d)$ is identified to the space of distributions $\div F$ with $F\in L^{p'}(\R^d; \IC^d)$. 
The space of continuous complex-valued functions on $\R$ that vanish at $\pm\infty$ is denoted by $C_{0}(\R)$ and $C_0([0,\infty))$ denotes the space of continuous functions on $[0,\infty)$ that vanish at $+\infty$.
All these spaces have  $E$-valued extension (denoted by $L^p(\R^d; E)$ and so on) when $E$ is a complex Banach space. Occasionally, we use the subscript $loc$ to indicate that certain conditions hold only on compact subsets. 

We denote points in $\ree=\R\times \R^{n}$ by $(t,x)$ etc. We set $\R^{1+n}_+:=(0,\infty)\times \R^n$. For short, we write $\Le u=-\div A \nabla u=0$ to mean \eqref{eq:divform}, where we  assume that the matrix    
\begin{equation}   
\label{eq:boundedmatrix}
  A(x)=(A_{i,j}^{\alpha,\beta}(x))_{i,j=0,\ldots, n}^{\alpha,\beta= 1,\ldots,m}\in L^\infty(\R^n;\cL(\IC^{m(1+n)})),
\end{equation} 
is bounded and measurable, independent of $t$ (transversal independence), and satisfies the following strict accretivity condition on  the subspace  $\cH$ of $L^2(\R^n;\IC^{m(1+n)})$ defined by $(f_{i}^\alpha)_{i=1,\ldots,n}$ being curl free in $\R^n$ for all $\alpha$:  For some $\lambda>0$ and all $f \in \cH$,
\begin{equation}   \label{eq:accrassumption}
   \int_{\R^n} \Re (A(x)f(x)\cdot  \con{f(x)}) \  dx\ge \lambda 
   \sum_{i=0}^n\sum_{\alpha=1}^m \int_{\R^n} |f_i^\alpha(x)|^2 \ dx.
\end{equation}
In particular situations, we may weaken this condition to the well-known G\aa{}rding inequality, see Remark~\ref{rem:ellipticity} below. The system \eqref{eq:divform} is considered in the sense of distributions with weak solutions in    $W^{1,2}_{loc}(\R^{1+n}_{+};\IC^m)$.

As weak solutions to elliptic systems might not be regular, we use the Whitney average variants of the usual non-tangential maximal functions. But when we get back to systems where solutions have meaningful pointwise values, these variants turn out to be equivalent to the usual pointwise control. Consider,  for $0<q<\infty$, the $q$-adapted non-tangential maximal function
\begin{equation}
\label{eq:KP}
 \NTq F(x):= \sup_{t>0}  \bigg(\bariint_{(c_0^{-1}t,c_0t)\times B(x,c_1t)} |F(s,y)|^q\ {ds \, dy}\bigg)^{1/q}, \qquad x\in \R^n,
\end{equation}
for  some fixed parameters $c_0>1$, $c_{1}>0$.  We use $B(x,r)$ for the Euclidean ball centred at $x$ with radius $r$ and denote averages by dashed integrals. When $q=2$, we simply  write $\NT$. For fixed $p>0$ and $q>0$, a covering argument reveals that changing the parameters yields equivalent $\|\NTq F\|_{p}$ norms.  In the following, we shall use \begin{equation}
\label{eq:whitney}
W(t,x):= (t/2,2t)\times B(x,t)
\end{equation}
for simplicity.

\subsection{Main results and consequences}
\label{sec:main results}

For  $1<p<\infty$, the $L^p$ Dirichlet  problem  with non-tangential maximal control can be formulated as follows:  given $f\in L^p(\R^{n}; \IC^m)$, uniquely solve
\begin{equation*}
 (D)_{p}^\Le  \qquad\qquad
\begin{cases}
  \Le u=0   & \text{on}\ \reu, \\
    \NT u\in L^{p}(\R^n),   \\
    \lim_{t \to 0} \bariint_{W(t,x)} |u(s,y)-f(x)|\ ds \, dy=0 & \text{for a.e. } x\in \R^n.\end{cases}
\end{equation*}
The $L^p$ regularity problem consists in solving uniquely (modulo constants), given $f\in \dot W^{1,p}(\R^n; \IC^m)$,  
\begin{equation*}
(R)_{p}^\Le  \qquad\qquad
\begin{cases}
   \Le u=0   & \text{on}\ \reu, \\
    \NT(\nabla u)\in L^{p}(\R^n)    \\
    \lim_{t \to 0} \bariint_{W(t,x)} |u(s,y)-f(x)|\ {ds \, dy}= 0 & \text{for a.e. } x\in \R^n.\end{cases}
\end{equation*}
We have fixed the parameters for $W(t,x)$ but from Lebesgue's differentiation theorem applied to $f$ and a covering argument we can again see that the convergence of Whitney averages of $|u-f|$ is independent of their particular choice.

To formulate our main results we implicitly use a certain perturbed first order operator $DB$ associated with $\Lop$ and the associated abstract Hardy spaces $H^p_{DB}$ defined and studied in \cite{AusSta}. At this stage, the reader need not be aware of their definitions as we are only going to use the conclusions drawn in that paper. 

There is an exponent $p_+(DB)>2$ that is related to resolvent estimates for $DB$ in $L^p$ and a coercivity property of $B$ and $B^*$ in $L^p$ and $L^{p'}$, respectively; the precise definition can be found in~\cite[Sec.~3.2]{AusSta}. It is also shown there that the set of those $p \in (\frac{n}{n+1}, p_+(DB))$ such that we have the coincidence $H^p_{DB}=H^p_{D}$ of abstract and concrete Hardy spaces is an interval and that $p_+(DB)$ is also its upper endpoint. This interval is called $I_L$ in~\cite{AM}. We call it $\mathcal{H}_\mathcal{L}$ in this article. This is an open interval containing $2$ and there is a corresponding interval $\cH_{\Le^*}$.

\begin{thm}\label{thm:uniqueDir} Let $1<p<\infty$ with  $p'\in \cH_{\Le^*}$. Existence for  $(R)_{p'}^{\Le^*} $   implies  uniqueness for $(D)_{p}^\Le$. 
\end{thm}
 
\begin{thm}\label{thm:uniquereg} Let  $1<p<\infty$ with $p\in \cH_{\Le}$. Existence for   $(D)_{p'}^{\Le^*} $   implies  uniqueness for $(R)_{p}^{\Le}$  (modulo constants).
\end{thm}

The interval  $\cH_{\Le}$ equals $(1-\varepsilon'(\Le), 2+\varepsilon(\Le))$ in case of (DGNM) for $\Le^*$ for example (with $0<\varepsilon(\Le)\le \infty$) and even some  conditions weaker than (DGNM) suffice, see Section 13 in \cite{AusSta} for details. We note that  \cite{HMaMo} uses a similar exponent $2+\varepsilon$ but we do not know whether it agrees with our $2+\varepsilon(\Le)$. More specifically, we have the following corollaries, compare with Proposition 8.19(i)\&(ii) in \cite{HMaMo}.
  
\begin{cor}\label{cor:uniqueDir} Assume (DGNM) for $\Le$ and  $(2+\varepsilon(\Le^*))'<p<\infty$. Existence for  $(R)_{p'}^{\Le^*} $   implies  uniqueness for $(D)_{p}^\Le$. 
\end{cor}
 
\begin{cor}\label{cor:uniquereg} Assume (DGNM) for $\Le^*$ and  $1<p<2+\varepsilon(\Le)$. Existence for   $(D)_{p'}^{\Le^*} $   implies  uniqueness for $(R)_{p}^{\Le}$   (modulo constants).
\end{cor}

Well-posedness of a boundary value problem is the conjunction of both existence of a solution for all data and uniqueness. A stronger notion, appearing implicitly in many earlier works, is that of compatible well-posedness: It means well-posedness and that the unique solution agrees with the energy solution obtained from the Lax-Milgram lemma, whenever the boundary data is admissible for energy solutions. Theorem~\ref{thm:uniqueDir} then has the following interesting consequence we shall discuss in detail in Section~\ref{sec:proofSF}. We define the square function $SF$ of a measurable function $F$ by   
\begin{equation}
\label{eq:sfdef}
 S F(x): = \left( \iint_{\Gamma_{a}(x)}  |F(t,y)|^2\ \frac{dt \, dy}{t^{n+1}}\right)^{1/2}, \qquad x\in \R^n,
\end{equation}
where $a>0$ is a fixed number called the aperture of the cone $\Gamma_{a}(x):= \{(t,y): t>0, |x-y|<at\}$.
 
\begin{cor}\label{cor:wp} Let $1<p<\infty$ with $p'\in \cH_{\Le^*}$. Assume   $(R)_{p'}^{\Le^*} $ is  well-posed (resp.\ compatible well-posed) modulo constants.  Then so is $(D)_{p}^\Le$. Moreover, given $f\in L^p(\R^n;\IC^m)$, the weak solution $u$ with data $f$ has further regularity $u \in C_{0}([0,\infty); L^p(\R^n; \IC^m))$, satisfies the square function estimate $\|S(t\nabla u)\|_{p}<\infty$ and there is comparability
\begin{equation} \label{eq:equiv}
\|\NT(u)\|_{p} \sim  \|S (t\nabla u)\|_{p} \sim \sup_{t\ge 0} \|u(t, \cdot )\|_{p}\sim \|f\|_{p}.
\end{equation} 
In addition, the non-tangential convergence improves to $L^2$ averages, that is, for a.e.\ $x\in \R^n$,
\begin{equation}
\label{eq:cvL2}
\lim_{t \to 0} \bariint_{W(t,x)} |u(s,y)-f(x)|^2\ {ds \, dy} = 0.
\end{equation}
\end{cor}

\subsection{Comparison to earlier results}

We comment here on the formulations of the problems and statements in relation to existing literature.

\begin{rem}[On convergence at the boundary for the Dirichlet problem]  \label{rem:boundary convergence}
 There is no trace theorem for the space of measurable functions $u$ with $\|\NT u\|_{p}<\infty$. Hence, existence of boundary values is part of the Dirichlet problem and does not follow from the interior control. If we look for non-tangential approach almost everywhere,  this the weakest possible condition. 
 But we may also choose a different convergence to the boundary data, such as 
 strong $L^p$ convergence $u(t, \cdot)\to f$ as $t \to 0$ on compact subsets of $\R^n$ as considered in \cite{HMaMo} under (DGNM). In that case, $\NT u$ can be replaced by the usual pointwise supremum on cones denoted by $N_{*}u$. As $\|u(t, \cdot)\|_{p}\le \|N_{*}u\|_{p}$ for all $p$ and $t>0$, the $L^p$ convergence on compact sets,  is also natural. We shall see that a minor modification of our arguments will cover this formulation of the Dirichlet problem and even a weaker form of $L^p_{loc}$ convergence. Note carefully that $\|\NT u\|_{p}<\infty$ does not imply $L^p$ boundedness for solutions when $p>2$ (see below).
\end{rem}

\begin{rem}[On the formulation of the Dirichlet problem]
 We use  $\NT = \widetilde{N}_{*,2}$ in the Dirichlet problem.  For complex equations, it makes a difference to consider $\NT$ or $\NTq$ with $q=p$, see \cite{May},  as solutions may not be locally $p$-integrable. The choice $q=2$ is most natural to overcome this difficulty and we could even use $\NTone$ by invoking reverse H\"older estimates. 
\end{rem}

\begin{rem}[On convergence at the boundary for the regularity problem]
 For the regularity problem, there is a trace theorem for the space of $L^2_{loc}$ functions satisfying $\|\NT(\nabla u)\|_{p}<\infty$, as is implicit in \cite{KP}. The Whitney averages converge almost everywhere in approaching the boundary,  the limit belongs to the homogeneous Sobolev space $\dot W^{1,p}(\R^n)$ and  Ces\`aro means $\barint_{t}^{2t} u \ ds$ converge in the sense of distributions modulo constants to the same limit, see Lemma~\ref{lem:trace} . Hence, the boundary condition in the regularity problem is implied by the interior control.  Actually, Theorem 1.1 of  \cite{AM} shows  that all solutions in this class for the range of $p$ in the statement enjoy convergence $ \nabla_{x} u(t,\cdot) \to  \nabla_{x} u(0,\cdot)$ strongly in $L^p$ as $t \to 0$. So, this could be taken as definition for the convergence to the boundary data as well. 
\end{rem}

\begin{rem}[On comparability of $S$ and $\NT$]
 For $p$ as in Corollary~\ref{cor:wp}, Theorems~1.6 and Theorem~1.9 in \cite{AM} show that (compatible) well-posedness for $(R)_{p'}^{\Le^*} $  is equivalent to (compatible) well-posedness for a variant $(\wt{D})_{p}^\Le$ of the Dirichlet problem with the non-tangential maximal function being replaced by the square function $S(t\nabla u)$ in the $L^p$-control. Owing to Corollary 1.4 in \cite{AM}, we have $\|\NT(u)\|_{p} \lesssim  \|S(t\nabla u)\|_{p}$ \emph{a priori} for any weak solution in this range of $p$ and $u(t,\cdot)$ converges to its boundary data strongly in $L^p$ if the right hand side is finite. Thus, Corollary~\ref{cor:wp} can rephrased as saying that the (compatible) well-posedness of $(\wt{D})_{p}^\Le$ implies that of $(D)_{p}^\Le$. 
 
 It would be interesting to prove the converse, at least for the range of $p$ above. For equations, that is $m=1$, with real valued $t$-independent coefficients, the real variable argument in \cite{HKMP1} shows $\|\NT(u)\|_{p} \sim  \|S(t\nabla u)\|_{p}$ for any weak solution. Hence, both Dirichlet problems are \textit{a priori} the same and the converse holds. Using the equivalence between $(\wt{D})_{p}^\Le$ and $(R)_{p'}^{\Le^*}$ mentioned above, this also provides a direct way for deducing the main result on well-posedness of the regularity problem $(R)_{p'}^{\Le^*}$ in \cite{HKMP2} for real coefficients from \cite{HKMP1}. For equations with complex coefficients and systems though, similar conclusions remain unknown.
\end{rem}

\begin{rem}[On representation by layer potentials]
 Another aspect of the theory is whether $u$ in Corollary~\ref{cor:wp} can be represented  as $u=\mD_{t}(\mD_{0^+})^{-1}f$, where $\mD_{t}$ is the double layer operator, also defined abstractly and proved to be bounded on $L^2$ in \cite{R} and on $L^p$ in this range of $p$ in \cite{AusSta}. There is no reason to believe that $\mD_{0^+}$ is invertible under the assumptions in Corollary~\ref{cor:wp}. Even well-posedness of the regularity problem on both half-spaces is not enough to conclude this: however, it gives the different representation $u=\mS_{t}(\mS_{0})^{-1}f$ using the single layer  operator. As of today, the only available method to prove invertibility  is via  the so-called Rellich estimates.  This was done first in \cite{V} when $p\ge 2$ for Laplace's equation in Lipschitz domains and has been extended to a larger class of equations (perturbations of real symmetric coefficients) in \cite{AAAHK} by developing the layer potential approach  and using the Rellich estimates of \cite{JK} for invertibility. Note that the Rellich estimates give access to solvability of  Neumann problems as well, which is strong additional information.    
\end{rem}

\begin{rem}[On the ellipticity condition]
\label{rem:ellipticity}
 Given $u \in \dot W^{1,2}(\ree)$, we can take $f(x) = \nabla u(t,x)$ for each $t \in \R$ in \eqref{eq:accrassumption} and integrate in $t$ to obtain G\aa{}rding's inequality
\begin{align} \label{eq:Ga}
 \iint_{\ree} \Re( A(x) \nabla u(t,x) \cdot \nabla u(t,x)) \ dx \, dt \geq \lambda \iint_{\ree} |\nabla u(t,x)|^2 \ dx \, dt.
\end{align}
We shall observe that our proofs of Theorem~\ref{thm:uniqueDir} and \ref{thm:uniquereg} in the case $p=2$ --- and even $p$ nearby --- only require \eqref{eq:Ga}. In particular, this gives access to uniqueness of boundary value problems for Lam\'e-type systems \cite{MMMM}, which typically satisfy G\aa{}rding's inequality but not the strict accretivity condition \eqref{eq:accrassumption}.
\end{rem}

\subsection{Strategy to the proofs}
The formal strategy is the same for both theorems and is adopted from earlier references, in particular \cite{AAAHK, HMaMo}. Let $u$ be a solution of $\Le u=0$ on $\reu$ with zero boundary condition. We take $G\in C_{0}^\infty(\ree; \IC^m)$ with support contained in some region $[a,b]\times B(0, c)$ contained in $\reu$. We want to show that $\dual u G =0$. 
We then pick a second function $\theta$ supported in $\reu$, real-valued, Lipschitz continuous  and equal to $1$ on the support of $G$. 
Finally, we let $H$ be a weak solution  of $\Le ^*H=G$ on $\ree$. As $u\theta$ is a test function for  this equation, we have
$$
\dual u G= \dual {u\theta} G = \dual {A\nabla (u\theta)} {\nabla H}.
$$
Next,
\begin{align*}
\label{}
\dual {A\nabla (u\theta)} {\nabla H}    & = \dual {Au\nabla \theta}{\nabla H} +  \dual {A \theta \nabla u }{\nabla H}  \\
    & =  \dual {Au\nabla \theta}{\nabla H} - \dual {A\nabla u }{H\nabla \theta  } + \dual {A\nabla u }{\nabla ( \theta  H)},
\end{align*}
and the last term vanishes because $\theta H$ is a test function for $\Le u=0$. All brackets here can be expressed by $L^2$ complex inner products and in accordance with our shorthand notation $-\div A \nabla u = 0$ for \eqref{eq:divform} we abbreviated $Au\nabla \theta = A_{i,j}^{\alpha, \beta} u^\beta \partial_j \theta$ and $H \nabla \theta = H^\alpha \partial_i \theta$, where sums are taken over repeated indices.

Now the existence hypothesis comes into play. We let $h:=H(0, \cdot)$ (provided it makes sense) and let 
$H_{1}$ be a solution to the adjoint problem $\Le ^*H_{1}=0 $ on $\reu$  with  boundary condition $h$. 
We may apply the same decomposition to $\dual {A\nabla (u\theta)} {\nabla H_{1}}$ and remark that this term vanishes since $u\theta$ is a test function for $\Le ^*H_{1}=0 $. Hence, we obtain
\begin{equation}
\label{eq:uG}
\dual u G= \dual {Au\nabla \theta}{\nabla (H-H_{1})} - \dual {A\nabla u }{  (H-H_{1})\nabla \theta}.
\end{equation}
We remark that $u$ and $H-H_{1}$ both vanish at the boundary. In fact, the reason to use $H_{1}$ is to help convergence near the boundary. The symmetry in $u$ and $H-H_{1}$ also indicates why the results can go both ways. 

The goal is then to show that these two terms tend to $0$ if we let $\theta\to 1$ everywhere on $\reu$. The heart of the matter is  to prove estimates on $H$ and $h$,   using our assumption, instead of relying on estimates for the fundamental solutions to represent $H$ in \cite{AAAHK, HMaMo} under (DGNM). For us, the assumption implies  boundedness properties of single layer operators for a certain range of spaces and we shall use this as a black box, once we have shown the representation \eqref{eq:L-1}.

Some particular choice of $\theta$ will facilitate the proofs. We are going to pick $\theta$ as follows. We fix $\chi\in C_{0}^\infty(\R^n)$ to be $1$ on $B(0,1)$ and with support in $B(0,2)$. We let $\eta$ be the continuous, piecewise linear function, which  is $0$ on $[0,2/3] $ and $1$ on $ [3/2, \infty)$ and linear in between. We pick $M> 2c$, $0<\varepsilon<a/4 $ and $2b< R<\infty$ to finally set 
\begin{align*}
\theta(t,x):= \chi(x/M)\eta(t/\varepsilon)(1-\eta(t/R)).
\end{align*}

\subsection{Standard estimates on weak solutions}
Here are some standard properties on weak solutions to $\Le u=0$ in a domain $\Omega \subset \ree$ we shall freely use throughout. The reader can refer to \cite{G} for the elliptic equations or to \cite{B} for systems. With regard to these references, we remark that reverse H\"older inequalities share the general feature that Lebesgue exponents on both sides can be lowered as one pleases, see Theorem~2 in \cite{IN} or Theorem~B.1 in \cite{BCF} for a particularly simple proof.

\medskip

\qquad Caccioppoli's inequality: $\displaystyle \bariint_{W(t,x)} |\nabla u|^2 \lesssim \frac{1}{t^2} \bariint_{\wt W(t,x)} | u|^2$.

\

\qquad Reverse H\"older inequality on $\nabla u$: $\displaystyle \bigg(\bariint_{W(t,x)} |\nabla u|^2\bigg)^{1/2} \lesssim  \bariint_{\wt W(t,x)} | \nabla u|$. 

\

\qquad Reverse H\"older inequality on $u$: $\displaystyle \bigg(\bariint_{W(t,x)} | u|^2\bigg)^{1/2} \lesssim  \bariint_{\wt W(t,x)} |  u|$. 

\

\noindent Here $ \wt W(t,x)$ is another Whitney region, with strictly larger parameters than $W(t,x)$ and we assume that the former is compactly included in $\Omega$. The implied constants depend only on ellipticity of $A$, Whitney parameters and the distance of $\wt W(t,x)$ to $\partial \Omega$. The reverse H\"older inequalities can also come with $L^p$-averages for some $p>2$ on the left but we shall not need such improvements.
\section{The case \texorpdfstring{$p=2$}{p=2}}\label{sec:p=2}

To illustrate the simplicity of our argument,  we present here the proofs of our main results in the case $p=2$. For the purpose of this section only, it will be sufficient to assume the weaker ellipticity condition \eqref{eq:Ga}.

\subsection{Estimate for \texorpdfstring{$\Le^{-1}$}{L-1}.}

First,  $\Le:\dot W^{1,2}(\ree; \IC^m) \to \dot W^{-1,2}(\ree;\IC^m)$ is invertible as a consequence of the Lax-Milgram lemma and \eqref{eq:Ga}. This is how we understand $\Le^{-1}$. The adjoint of $\Le$ is associated with the matrix $A^*$. 

\begin{lem}\label{lem:L-1} Let $G\in C_{0}^\infty(\ree; \IC^m) \cap \dot W^{-1,2}(\ree; \IC^m)$ and let $H\in \dot W^{1,2}(\ree; \IC^m)$ solve $\Le H=G$ in $\ree$. Then 
\begin{enumerate}
  \item For all integers $k\ge 1$, $\partial_{t}^kH$ exists in $W^{1,2}(\ree; \IC^m)$ and $\Le(\partial_{t}^kH)= \partial_{t}^kG$. 
  \item For all integers $k\ge 1$, $\partial_{t}^kH \in C_{0}(\R; W^{1,2}(\R^n; \IC^m))$ with $t \in \R$ the distinguished variable.
\end{enumerate} 
\end{lem}

\begin{proof} Note that for $n\ge 3$ we have $C_{0}^\infty(\ree; \IC^m) \subset \dot W^{-1,2}(\ree; \IC^m)$ by Sobolev embeddings. For $n=1,2$, the necessary and sufficient condition on $G$ is $\iint_{\ree} G=0$. The solution $H\in \dot W^{1,2}(\R^n;\IC^m)$ is defined by requiring for all $\varphi \in \dot W^{1,2}(\ree; \IC^m)$ that
 \begin{equation}
\label{eq:weaksol}
\dual {A\nabla H}{\nabla \varphi}= \dual G \varphi,
\end{equation}
where the second bracket is the (complex) duality between  $\dot W^{-1,2}(\ree; \IC^m)$ and $\dot W^{1,2}(\ree; \IC^m)$.
As $A$ is $t$-independent, and since $\partial_{t}^kG \in \dot W^{-1,2}(\ree; \IC^m)$  for all $k\ge 1$, the method of difference quotients and induction on $k$ allows us to differentiate \eqref{eq:weaksol} and to obtain $ \partial_{t}^kH \in \dot W^{1,2}(\ree; \IC^m)$ with
$$ 
 \dual {A\nabla  \partial_{t}^kH}{\nabla \varphi}= \dual  {\partial_{t}^kG} \varphi,
$$
for all $\varphi \in \dot W^{1,2}(\ree; \IC^m)$. This means $\Le(\partial_{t}^kH)= \partial_{t}^kG$. Moreover, $\nabla (\partial_{t}^kH) \in L^2(\ree, \IC^{m(1+n)})$ for all integers $k\ge 0$, showing in particular $\partial_{t}^{k+1}H \in L^2(\ree; \IC^m)$ for all $k\ge 0$. This completes the proof of (i).
 
For (ii), we use the vector-valued embedding $W^{1,2}(\R; L^2(\R^n; \IC^m)) \subset C_{0}(\R; L^2(\R^n; \IC^m))$. Since $\partial_{t}^kH,  \partial_{t}^{k+1}H$ are in $L^2(\ree; \IC^m)$, which we identify with $L^2(\R; L^2(\R^n; \IC^m))$ via Fubini's theorem, we obtain $\partial_{t}^kH \in C_{0}(\R; L^2(\R^n; \IC^m))$. Similarly we have $\nabla_{x} \partial_{t}^k H,  \nabla_{x} \partial_{t}^{k+1}H \in L^2(\ree; \IC^{mn})$ and hence $\nabla_{x} \partial_{t}^k H \in C_{0}(\R; L^2(\R^n; \IC^{mn}))$ as well. The conclusion follows. 
\end{proof}

\begin{lem}
\label{lem:NTH} 
Let $\wt G \in C_{0}^\infty(\ree; \IC^m) \cap \dot W^{-1,2}(\ree; \IC^m)$ and $\wt H: =\Le ^{-1}(\wt G)$. Set $G=\partial_{t}\wt G$ and $H=\partial_{t}\wt H$.  Then 
$G\in C_{0}^\infty(\ree; \IC^m) \cap \dot W^{-1,2}(\ree; \IC^m)$ and $H\in  W^{1,2}(\ree; \IC^m)$ solves $\Le H=G$ in $\ree$ with estimates 
\begin{align*}
 \|\NTone H\|_{2}+ \|\NTone(\nabla H)\|_{2}<\infty.
\end{align*}
\end{lem}

\begin{rem}
With a little more work the reader may check $\|\NT H\|_{2}< \infty$ and $\|\NT(\nabla H)\|_{2}<\infty$. We do not need this improvement.
\end{rem}

\begin{proof}  
As a derivative of an $L^2(\ree; \IC^m)$-function, $\partial_t \wt{G}$ is in  $\dot W^{-1,2}(\ree;\IC^m)$. Observe that  by Lemma~\ref{lem:L-1} we have 
\begin{align*}
 H=\partial_{t}\wt H= \Le ^{-1}(\partial_{t}\wt G)=\Le ^{-1}(G)\in W^{1,2}(\R^n;\IC^m).
\end{align*}
In particular, $H$ is square-integrable. Let $a,b\in \R$ such that $\supp G \subset [a,b] \times \R^n$. We may assume for simplicity $b\ge 2$.

To establish $\|\NTone H\|_{2}< \infty$, we split  the supremum defining  $\NTone H$ in two parts according to $t<4b$  and $4b\le t$. In the first case we note that $h:=H(0,\cdot)$ is defined in $W^{1,2}(\R^n;\IC^m)$ due to Lemma~\ref{lem:L-1} to give
\begin{align*}
 \bariint_{W(t,x)} |H(s,y)|\ {ds \, dy} \le     \bariint_{W(t,x)} |H(s,y)-h(y)|\ {ds \, dy}+  \bariint_{W(t,x)} |h(y)|\ {ds \, dy}. 
\end{align*}
Since $s\mapsto H(s,\cdot)$ is smooth with values in $L^2(\R^n; \IC^m)$ again by Lemma \ref{lem:L-1}, and as $t<4b$, we can write 
\begin{align}
\label{eq1:NTH}
\begin{split}
 \bariint_{W(t,x)} |H(s,y)|\ {ds \, dy} 
&\le \barint_{B(x,t)} \int_{0}^{8b} |\partial_{s}H(s,y)|\ ds \, dy + \barint_{B(x,t)} |h(y)|\ {ds \, dy} \\
&\le \HL(F)(x)+ \HL(h)(x),
\end{split}
\end{align}
where $\HL$ is the Hardy-Littlewood maximal operator on $\R^n$ and $F(x) := \int_{0}^{8b} |\partial_{s}H(s,x)|\ ds$. We know that $h\in L^2(\R^n; \IC^m)$, and also $F\in L^2(\R^n)$ since  $\partial_{s}H\in L^2(\ree; \IC^m) $ and
\begin{align*}
     \int_{\R^n} |F(x)|^2\ dx \le 8b \int_{\R^n} \int_{0}^{8b} |\partial_{s}H(s,x)|^2\ ds \, dx.
\end{align*}
Taking the supremum over $t<4b$ in \eqref{eq1:NTH}, we obtain the $L^2$ bound from the maximal theorem. Assume now that $t\ge 4b$. Then  for $T\ge 2t$,
\begin{align*}
 \bariint_{W(t,x)} |H(s,y)|\ {ds \, dy}   \le \barint_{B(x,t)}   \int_{2b}^{T} |\partial_{s}H(s,y)|\ ds \, dy + \barint_{B(x,t)} |H(T,y)|\ dy.
\end{align*}
Applying Lemma \ref{lem:L-1}(ii) to $\partial_{t}\wt H= H$, we see that the second term on the right-hand side tends to $0$ as $T\to \infty$. Thus, 
\begin{align}
\label{eq2:NTH}
 \sup_{t\ge 4b}\bariint_{W(t,x)} |H(s,y)|\ {ds \, dy} \le M (F_{1})(x), 
\end{align}
with $F_{1}(x) := \int_{2b}^{\infty} |\partial_{s}H(s,x)|\ ds$. Now, 
\begin{align*}
 \int_{\R^n} |F_{1}(x)|^2\ dx \le \frac{1}{2b}  \int_{\R^n} \int_{2b}^\infty s^2|\partial_{s}H(s,x)|^2\ ds \, dx
\end{align*}
and in the domain of integration, $H$ is a weak solution to $\Le H=0$. Thus, covering this region by Whitney cubes for $\reu$, that is, cubes having sidelength half their distance to the boundary,  we may apply Caccioppoli's inequality on each cube and sum up, using bounded overlap, to get 
\begin{align*}
 \int_{\R^n} \int_{2b}^\infty s^2|\partial_{s}H(s,x)|^2\ ds \, dx \lesssim \int_{\R^n} \int_{b}^\infty |H(s,x)|^2\ ds \, dx <\infty.
\end{align*}
Going back to \eqref{eq2:NTH}, the claim follows again from the maximal theorem.

We next turn to establishing  $\|\NTone (\nabla H)\|_{2}< \infty$. The control for the $t$-derivative $\partial_{t}H$ is the same upon replacing $H$ by $\partial_{t}H$ and $\wt{G}$ by $\partial_t \wt{G}$ in the argument above, which satisfy the same hypotheses. Let us turn to $\nabla_{y}H$. Again we split the supremum into two parts $t<4b$  and $4b\le t$. As for the first one, we argue as in \eqref{eq1:NTH}, using that $s\mapsto \nabla_{y}H(s,\cdot)$ is smooth with values in $L^2(\R^n; \IC^{mn})$ by Lemma~\ref{lem:L-1}, to give
\begin{align}
\label{eq4:NTH}
  \bariint_{W(t,x)} |\nabla_{y}H(s,y)|\ {ds \, dy}  \le \HL(\wt F)(x)+ \HL(\nabla h)(x),
\end{align}
where $\wt F(x):=  \int_{0}^{8b} |\partial_{s}\nabla_{x}H(s,x)|\ ds$. We know that $\nabla h\in L^2(\R^n; \IC^{mn})$, and also $\wt F\in L^2(\R^n)$ since by  Lemma~\ref{lem:L-1},  
\begin{align*}
     \int_{\R^n} |\wt F(x)|^2\ dx \le 8b\int_{\R^n} \int_{0}^{8b} |\partial_{s}\nabla_{x}H(s,x)|^2\ ds \, dx <\infty.
\end{align*}
Taking the supremum over $t<4b$ in \eqref{eq4:NTH}, we obtain again an $L^2$ bound from the maximal theorem. If $t\ge 4b$, we argue as in \eqref{eq2:NTH} to find
\begin{align}
\label{eq3:NTH}
 \sup_{t\ge 4b}\bariint_{W(t,x)} |\nabla_{y}H(s,y)|\ {ds \, dy} \le \HL (\wt F_{1})(x), 
\end{align}
with $\wt F_{1}(x):= \int_{2b}^{\infty} |\partial_{s}\nabla_{x} H(s,x)|\ ds$.
Next, we can use the same covering argument as before to bring into play Caccioppoli's inequality and deduce 
\begin{align*} 
 \int_{\R^n} |\wt F_{1}(x)|^2\ dx \le \frac{1}{2b}  \int_{\R^n} \int_{2b}^\infty s^2|\partial_{s}\nabla_{x}H(s,x)|^2\ ds \, dx \lesssim \int_{\R^n} \int_{b}^\infty |\partial_{s}H(s,x)|^2\ ds \, dx.
\end{align*}
Lemma~\ref{lem:L-1} guarantees that the rightmost term is finite and a final application of the maximal theorem yields the $L^2$ bound in \eqref{eq3:NTH}.
\end{proof}

\subsection{Proof of Theorem \ref{thm:uniqueDir} when \texorpdfstring{$p=2$}{p=2}}

\label{Thm1 p=2}
We assume $\Le u=0$ on $\reu$, the control $\NT u \in L^2(\R^n)$ and we have the convergence 
\begin{align}
\label{eq:NT 0}
 \lim_{t \to 0} \bariint_{W(t,x)} |u(s,y)|\ {ds \, dy} = 0
\end{align}
for almost every $x \in \R^n$. We have to show $u=0$ almost everywhere. To this end, we apply the strategy presented in Section \ref{sec:strategy}.

For a reason  that will appear later in the proof,  we pick $G$ of the form $G=\partial_{t}\wt G$ with $\wt G \in C_{0}^\infty(\ree; \IC^m) \cap \dot W^{-1,2}(\ree; \IC^m)$.  Assume, we have already proved $\dual u G=0$. This means $\dual {\partial_{t}u} {\wt G}=0$. When $n\ge 2$, Sobolev embeddings show that $\wt G$ can be any test function and so this implies $u(t,x)=f(x)$. When $n=1$, we can take any test function with zero average and obtain $u(t,x)=ct+f(x)$ with $c$ constant. The equations hold a.e.\ and we have $f \in L^2_{loc}(\reu)$ since $u \in L^2_{loc}(\reu)$. Due to the limit of Whitney averages at $t=0$, we obtain $f=0$ a.e.\ in both cases by Lebesgue's differentiation theorem. When $n\ge 2$, we are done. When $n=1$, this yields $u(t,x)=ct$, hence $\bariint_{W(t,x)} |u(s,y)|^2\ {ds \, dy} = \frac{5}{4} c^2t^2 $.
As the supremum in $t>0$ is finite a.e., we must have $c=0$. 

To show $\dual u G =0$, we have to make sense of $H_1$ and control both terms on the right-hand side of \eqref{eq:uG}.
We let $\wt H := (\Le ^{*})^{-1}(\wt G)$ and have $H=\partial_{t}\wt H$ due to Lemma~\ref{lem:L-1}(i).  As $\nabla h\in L^2(\R^n; \IC^{mn})$ by Lemma~\ref{lem:L-1}(ii), existence in the regularity problem for $\Le ^*$ yields a solution $H_{1}$ to $\Le ^*H_{1}=0$ in $\reu$ with $\NT (\nabla H_{1}) \in L^2(\R^n)$ and boundary trace $h$. Due to the explicit form of $\theta$, we easily obtain for the first integral in \eqref{eq:uG},
$$
|\dual {Au\nabla \theta}{\nabla (H-H_{1})}| \lesssim  I_{M}+ J_{\varepsilon}+J_{R},
$$
with  
$$I_{M} := \frac{1}{M} \int_{|x| \ge M} \int_{2\varepsilon/3}^{3R/2} |u||\nabla (H-H_{1})|\ ds \, dy$$
and 
$$
J_{\alpha} := \int_{\R^n} \barint_{2\alpha/3}^{3\alpha/2} |u||\nabla (H-H_{1})|\ ds \, dy.
$$

First, $I_{M}$  tends to $0$ as $M\to \infty$. Indeed, let $\Omega:=[\frac{2\varepsilon}{3}, \frac{3R}{2}]\times \{|x|\ge M\}$. By Lemma~\ref{lem:product}, 
\begin{align*}
     I_M \lesssim \frac{1}{M} \|\NT u\|_{2}\|\NT(1_{\Omega}|\nabla (H-H_{1})|)\|_{2}.
\end{align*} 
As  $H-H_{1}$ is a solution to $\Le ^*(H- H_{1})=0$ on a neighbourhood of $\Omega$, we can use reverse H\"older inequalities and a change of Whitney parameters to obtain 
\begin{align*}
 \|\NT(1_{\Omega}|\nabla (H-H_{1})|)\|_{2} \lesssim \|\NTone(\nabla(H-H_{1}))\|_{2},
\end{align*}
which is finite by Lemma~\ref{lem:NTH} and the construction of $H_{1}$. 

Next, set  $w(\varepsilon,x) :=(\frac{2\varepsilon}{3}, \frac{3\varepsilon}{2})\times B(x,\frac{\varepsilon}{2})$, which is a Whitney region compactly contained in $W(\varepsilon,x)$. Using the averaging trick with balls of $\R^n$ having radii $\varepsilon/2$ and Tonelli's theorem, we obtain
\begin{align*}
J_{\varepsilon}    & \leq \int_{\R^n}  \bigg(\bariint_{w(\varepsilon,x)} |u||\nabla (H-H_{1})|\bigg)\ dx   \\
    &  \lesssim \int_{\R^n} \bigg(\bariint_{w(\varepsilon,x)} |u|^2 \bigg)^{1/2}\bigg(\bariint_{w(\varepsilon,x)} |\nabla (H-H_{1})|^2 \bigg)^{1/2}\ dx \\
    & \lesssim \int_{\R^n} \bigg(\bariint_{W(\varepsilon,x)} |u| \bigg)\bigg(\bariint_{ W(\varepsilon,x)} |\nabla (H-H_{1})| \bigg)\ dx,
\end{align*} 
where we have used the reverse H\"older inequality for $u$ and $\nabla (H-H_{1})$.  Observe that the integrand is controlled by $\NT u \cdot \NTone(\nabla(H-H_{1}))$, which as a product of two $L^2$-functions is integrable. Moreover, $\bariint_{W(\varepsilon,x)} |u|$ tends to $0$ as $\varepsilon\to 0$ by assumption. Thus, we conclude $J_{\varepsilon}\to 0$ as $\varepsilon\to 0$ by dominated convergence. 
Finally, we have similarly
\begin{align*}
J_{R}  \lesssim \int_{\R^n} \bigg(\bariint_{ W(R,x)} |u| \bigg)\bigg(\bariint_{ W(R,x)} |\nabla (H-H_{1})| \bigg)\ dx,
\end{align*}
so that we get the same $L^1$-control, while $\bariint_{W(R,x)} |u|\to 0$ as $R\to \infty$ follows from Lemma~\ref{lem:infinity} and $ \NT(u)\in L^2(\R^n)$. 

Finally, we treat $\dual {A\nabla u   }{(H-H_{1})\nabla \theta}$, which is the other term on the right-hand side of \eqref{eq:uG}, exactly as above upon replacing  $|u||\nabla (H-H_{1})|$ by 
$|\nabla u| |H-H_{1}|= (s|\nabla u|) (|H-H_{1}|/s)$. We observe that Caccioppoli's inequality and the reverse H\"older inequality reveal
\begin{align*}
\bigg(\bariint_{w(t,x)} s^2|\nabla u|^2\bigg)^{1/2} \lesssim \bariint_{W(t,x)} |u|.
\end{align*}
Hence we can apply the same argument as before, using $\|\NTone ((H-H_{1})/t)\|_{2} \in L^2(\R^n)$ from Lemma~\ref{lem:trace} and $H=H_{1}$ on the boundary (write $H-H_{1}=H-h+h-H_{1}$), which is used for the first time here. 

\begin{rem}\label{rem:modif}
With regard to Remark~\ref{rem:boundary convergence}, we give the modification of the proof when instead of the non-tangential convergence \eqref{eq:NT 0} we assume $\barint_{t/2}^{2t} |u(s,\cdot)|\, ds \to 0$ as $t\to 0$ in $L^2_{loc}(\R^n)$. The only difference is in the treatment of the limit of $J_{\varepsilon}$. To this end, pick any $\delta>0$ and choose $r>0$ such that
\begin{align*}
 \int_{{}^cB(0,r)} \NT(u) \NTone(\nabla(H-H_{1})) \ dx <\delta.
\end{align*}
Then for $\varepsilon<1$ we obtain from Tonelli's theorem
  \begin{align*}
\label{}
 J_{\varepsilon} & 
\lesssim  
\delta+ \int_{B(0,r)} \bigg(\bariint_{W(\varepsilon,x)} |u(s,y)| \ ds \, dy \bigg)
\NTone(\nabla(H-H_{1}))(x)\ dx  
 \\
    &  \lesssim
\delta+\int_{B(0,r+1)} \bigg(\barint_{\varepsilon/2}^{2\varepsilon} |u(s,y)| \ ds\bigg) \bigg(\barint_{B(y,\varepsilon)} 
\NTone(\nabla(H-H_{1}))(x)\ dx \bigg)  \ dy \\
    &  \leq
\delta+ \bigg\|\barint_{\varepsilon/2}^{2\varepsilon} |u(s,\cdot)| \ ds\bigg\|_{L^2(B(0,r+1))}  \bigg\|\HL(\NTone(\nabla(H-H_{1})))\bigg\|_{L^2(\R^n)},
    \end{align*}
which in the limit superior $\varepsilon\to 0$ is bounded by $\delta$ using the hypothesis and the maximal theorem. The modifications for the integral involving $(s|\nabla u|) (|H-H_{1}|/s)$ are similar, incorporating Caccioppoli's inequality to get back to averages of $u$. 
\end{rem}

\subsection{Proof of Theorem \ref{thm:uniquereg} when \texorpdfstring{$p=2$}{p=2}}
\label{Thm2 p=2}

We assume $\Le  u=0$ on $\reu$ with $\NT (\nabla u) \in L^2(\R^n)$ and convergence $\lim_{t \to 0} \bariint_{W(t,x)} |u(s,y)|\ {ds \, dy} = 0$ for a.e.\  $x\in \R^n$. We have to show $u=0$ almost everywhere. We first remark that by Lemma \ref{lem:trace} we also have $\NTone(u/t)\in L^2(\R^n)$.

As in the proof of Theorem~\ref{thm:uniqueDir}, we pick $G$ of the form $G=\partial_{t}\wt G$ with $\wt G \in C_{0}^\infty(\ree; \IC^m) \cap \dot W^{-1,2}(\ree; \IC^m)$. We claim that it is again enough to show $\dual u G =0$: Indeed, when $n\ge 2$ we may conclude as before and when $n=1$, we reach the point where $u(t,x)=ct$ a.e.\ but as $\NT(u/t)\in L^2(\R^n)$ we must have $c=0$. 

To actually show $\dual u G =0$, we have again to control both terms on the right hand side of \eqref{eq:uG}. We let $\wt H := (\Le ^{*})^{-1}(\wt G)$, noting that by Lemma \ref{lem:L-1} we have $H=(\Le ^{*})^{-1}(G)=\partial_{t}\wt H$.
As $ h\in L^2(\R^n; \IC^m)$ by Lemma~\ref{lem:L-1}(ii), existence for the Dirichlet problem yields a solution $H_{1}$ to $\Le ^*H_{1}=0$ in $\reu$ with $\NT (H_{1}) \in L^2(\R^n)$ and boundary trace $h$ in the sense that $\bariint_{W(\varepsilon,x)} |H_{1}(s,y)-h(x)|\to 0$ for a.e.\ $x \in \R^n$ as $\varepsilon\to 0$. We are now ready to estimate the integrals in \eqref{eq:uG}. We still have
$$
|\dual {Au\nabla \theta}{\nabla (H-H_{1})}| \lesssim  I_{M}+ J_{\varepsilon}+J_{R},
$$
with  
$$I_{M} = \frac{1}{M} \int_{|x|\ge M} \int_{2\varepsilon/3}^{3R/2} \frac{|u|}{s} \cdot |s\nabla (H-H_{1})|\ ds \, dy$$
and 
$$
J_{\alpha}=  \int_{\R^n} \barint_{2\alpha/3}^{3\alpha/2} \frac{|u|}{s} \cdot |s\nabla (H-H_{1})|\ ds \, dy.
$$

First, $I_{M}$ tends to $0$ as $M\to \infty$. Indeed, let $\Omega :=[2\varepsilon/3, 3R/2]\times \{|x|\ge M\}$. By Lemma~\ref{lem:product},
\begin{align*}
 I_M \lesssim \frac{1}{M}\|\NT( u/s)\|_{2}\|\NT(1_{\Omega}|s\nabla (H-H_{1})|)\|_{2}.
\end{align*}
As $\Lop u=0$ and $s\sim t$ on Whitney regions $W(t,x)$, we have $\|\NT( u/s)\|_{2} \lesssim \|\NTone( u/s)\|_{2} < \infty$ by reverse H\"older estimates and change of parameters (which will be implicit in all of the following steps). Also $\Le ^*(H- H_{1})=0$ holds on a neighbourhood of $\Omega$. Thus, we can use Caccioppoli inequalities to obtain $\|\NT(1_{\Omega}s|\nabla (H-H_{1})|)\|_{2} \lesssim \|\NT(1_{\wt \Omega}(H-H_{1}))\|_{2}$, where $\wt\Omega$ is a slightly  bigger region, still at some large distance from  the support of $\wt G$, so that we may use reverse H\"older inequalities to conclude $\|\NT(1_{\wt \Omega}(H-H_{1}))\|_{2} \lesssim \|\NTone H\|_{2}+ \|\NT H_{1}\|_{2}$. The latter are finite by Lemma \ref{lem:NTH} and the construction of $H_{1}$. 

Next, using the averaging trick with balls of radii $\varepsilon/2$ and the Whitney regions $w(\varepsilon,x):=(\frac{2\varepsilon}{3}, \frac{3\varepsilon}{2})\times B(x,\frac{\varepsilon}{2})$ and $\wt w(\varepsilon,x)=(\frac{4\varepsilon}{7}, \frac{7\varepsilon}{4}) \times B(x,\frac{2\varepsilon}{3})$, both compactly contained in $W(\varepsilon,x)$, we obtain
\begin{align*}
J_{\varepsilon}    & \lesssim \int_{\R^n} \bigg(\bariint_{w(\varepsilon,x)} \frac{|u|}{s} \cdot |s\nabla (H-H_{1})|\bigg)\ dx   \\
    &  \lesssim \int_{\R^n} \bigg(\bariint_{w(\varepsilon,x)} \frac{|u|^2}{s^2} \bigg)^{1/2}\bigg(\bariint_{w(\varepsilon,x)} |s\nabla (H-H_{1})|^2 \bigg)^{1/2}\ dx \\
    & \lesssim \int_{\R^n} \bigg(\bariint_{W(\varepsilon,x)} \frac{|u|}{s} \bigg)\bigg(\bariint_{ \wt w(\varepsilon,x)} |H-H_{1}|^2 \bigg)\ dx
    \\
    &  \lesssim \int_{\R^n} \bigg(\bariint_{W(\varepsilon,x)} \frac{|u|}{s} \bigg)\bigg(\bariint_{ W(\varepsilon,x)} |H-H_{1}| \bigg)\ dx,
\end{align*} 
where we have used the reverse H\"older inequality for $u$ and $H-H_{1}$, and Caccioppoli inequalities for $H-H_{1}$, observing that $\Le^*(H-H_{1}) = 0$ holds on a neighbourhood of the domain of integration.  We note that the integrand is controlled by $\NT (u/s) \NTone(H-H_{1})$, which is integrable by assumption on $u$ and Lemma~\ref{lem:trace} plus Lemma~\ref{lem:NTH} for $H$ and the construction of $H_{1}$. As for the pointwise convergence, we use $H-H_{1}$ and write 
\begin{align*}
\bariint_{W(\varepsilon,x)} |H-H_{1}| & \le \bariint_{W(\varepsilon,x)} |H(s,y)-h(x)|\ ds \, dy  + \bariint_{W(\varepsilon,x)} |H_{1}(s,y)-h(x)|\ ds \, dy.
\end{align*}
Letting $\varepsilon\to 0$, the first term goes to $0$ by Lemma~\ref{lem:NTH} combined with Lemma~\ref{lem:trace}. By construction of $H_{1}$, so does the second one. Thus, $J_{\varepsilon}\to 0$ as $\eps \to 0$ by dominated convergence. 
Finally, we have similarly, 
\begin{align*}
J_{R}  \lesssim \int_{\R^n} \bigg(\bariint_{ W(R,x)} \frac{|u|}{s} \bigg)\bigg(\bariint_{ W(R,x)} |H-H_{1}| \bigg)\ dx,
\end{align*}
so that we get the same $L^1$-control and to obtain convergence to $0$ we can use $\bariint_{W(R,x)} |H-H_{1}|\to 0$ as $R\to \infty$, which follows from Lemma~\ref{lem:infinity} since $\NTone(H-H_{1})\in L^2(\R^n)$. 

Finally, we can treat $\dual {A\nabla u   }{(H-H_{1})\nabla \theta}$, which is the second term on the right-hand side of \eqref{eq:uG}, as above upon replacing  $(|u|/s)(s|\nabla (H-H_{1})|)$ by 
$|\nabla u| |H-H_{1}|$. The arguments  for convergences when $\varepsilon\to 0$ and $R\to \infty$ are almost identical and we leave the details to the reader. 
\section{Representation by single layer operators}
\label{sec:layer}

In order not to disrupt the flow of the proofs of uniqueness,  we only summarise here the needed results for the single layer  operators of \cite{R}, which we denote  by $\mS_{t}^{\Le }$, postponing proofs until Section~\ref{sec:prooflayer}. Throughout, we use the notation $f_{s}\colon x\mapsto f(s,x)$. We begin with the  result summarizing their boundedness properties. 

\begin{lem}\label{lem:bl}   Let $1<p<\infty$. Let $k \geq 0$ be an integer.  Set $\R^*=\R\setminus \{0\}$.
\begin{enumerate}
  \item If   $p\in \cH_{\Le }$, we have the mapping properties (of extensions by density of)
\begin{align*}
(t\partial_{t})^k\mS_{t}^{\Le }:L^{p}(\R^n; \IC^m) \to \dot W^{1,p}(\R^n; \IC^m), \qquad t\in \R^* 
\end{align*}
and 
\begin{align*}
(t\partial_{t})^k\partial_{t}\mS_{t}^{\Le }: L^{p}(\R^n; \IC^m)\to L^{p}(\R^n; \IC^m), \qquad t\in \R^*.   
\end{align*}
Moreover, the operator norms are uniform with respect to $t$, and as functions of $t$,  these  operators are strongly continuous  on $\R^*$, have strong limits at  $0^\pm$  and vanish strongly at infinity. Moreover, the limits at $0^\pm$ are the same except for $\pd_{t}\mS_{t}^{\Le}$ on $L^p$.

  \item  If $p'\in \cH_{\Le ^*}$, 
we have the mapping properties (of extensions by density of)
\begin{align*}
(t\partial_{t})^k\mS_{t}^{\Le }:\dot W^{-1,p}(\R^n; \IC^m) \to  L^{p}(\R^n; \IC^m), \qquad t\in \R^*. 
\end{align*}
Moreover, the operator norms are uniform with respect to $t$, and as functions of $t$,  these  operators are  strongly continuous  on $\R^*$, have strongly continuous extensions  at  $0$  and vanish strongly at infinity.    
\end{enumerate}  
\end{lem}

\begin{rem} Recall that the intervals $\cH_{\Le }, \cH_{\Le ^*}$ were defined in Section~\ref{sec:main results}.
They are open intervals and contain 2, so $\cH_{\Le }
\cap (\cH_{\Le ^*})'$, where $I'=\{p'; p\in I\}$, is an open interval around 2 and for $p$ in this interval both sets of estimates hold.   
\end{rem}

\begin{rem} In the first case, it follows from the stated properties that $(t\partial_{t})^k\nabla_{x}\mS_{t}^{\Le }$ are $L^p$ bounded uniformly in $t\ne 0$. The same is true for the conormal derivatives $(t\partial_{t})^k\dnuA\mS_{t}^{\Le }$ and these operators are not continuous at 0 when $k=0$ (jump relations). In the second case, these operators are $\dot W^{-1,p}$ bounded uniformly in $t$. 
\end{rem}

\begin{rem}
More mapping properties on fractional Sobolev and Besov spaces can be drawn by interpolating these two sets of inequalities. We do not need those here. Note that for $1<p<\infty$, the condition $p'\in \cH_{\Le ^*}$ is closely related to the identification of certain negative order Sobolev spaces $\dot W^{-1,p}_{DB}=\dot W^{-1,p}_{D}$ as $\dot W^{-1,p}_{D}$  is a natural space for $(\dnuA u|_{t=0}, \nabla_{x}f)$ if $u$ is a solution to $\Le u = 0$ with Dirichlet data $f\in L^p$. The reader can refer to \cite{AA} for more on this issue. 
\end{rem}

Next, we state the representation formula by the above single layer operators for $\Le ^{-1}$. The proof  will only use the properties  stated above for $p=2$. 

\begin{prop}[Representation by single layer operators]
\label{prop:representation}
Assume $f\in C_{0}^\infty(\ree;\IC^m)$. If $n=1,2$ assume furthermore that $f=\div_{x}F$ for some $F\in C_{0}^\infty(\ree; \IC^{mn})$. Then 
\begin{align}
\label{eq:cvgamma1}
(\Le ^{-1}f)(t,x)
= \mathrm{p.v.}\!\int_{\R} \mS_{t-s}^\Le  f_s(x) \ ds, \quad \mathrm{in }\  \dot W^{1,2}(\ree; \IC^m),
\end{align}
where 
\begin{equation*}
\mathrm{p.v.}\!\int_{\R} \mS_{t-s}^\Le  f_s(x) \ ds := \lim_{\varepsilon\to 0, R\to \infty}\int_{\varepsilon<|t-s|<R} \mS_{t-s}^\Le  f_s(x) \ ds, 
\end{equation*} 
and
\begin{align}
\label{eq:cvgamma2}
\begin{split}
\Le ^{-1}(\partial_{t}f)(t,x) 
=  (\partial_{t} \Le ^{-1}f)(t,x) 
&= \mathrm{p.v.}\!\int_{\R} \partial_{t}\mS_{t-s}^\Le  f_{s}(x)\ ds \\
&=\mathrm{p.v.}\!\int_{\R} \mS_{t-s}^\Le  (\partial_{s}f_{s})(x) \ ds,   \quad \mathrm{in }\   L^2(\ree; \IC^m).
\end{split}
\end{align}
Furthermore, with $W^{1,2}(\R^n; \IC^m)$-convergence uniformly for $t\in \R$, we have norm convergent (Bochner) integrals in $W^{1,2}(\R^n; \IC^m)$,
\begin{equation}
\label{eq:cvgamma3}
(\Le ^{-1}(\partial_{t}f))_{t}=  (\partial_{t} \Le ^{-1}f)_{t}=  \int_{\R} \partial_{t}\mS_{t-s}^\Le  f_{s}\ ds=  \int_{\R} \mS_{t-s}^\Le  (\partial_{s}f_{s})\ ds.
\end{equation}
\end{prop}

\begin{rem}   It is classical that $f_t(x)=\div_{x}F_t(x)$ is equivalent to $\int_{\R^n} f_t(x) \ dx =0$. 
 We note that the replacement of $f$ by $\partial_{t}f$, which has all the other properties required for $f$, yields a better convergence  of the principal value towards $\Le ^{-1}(\partial_{t}f)$. Also for fixed $t$, the integrals in  \eqref{eq:cvgamma3} have good behaviour: It is only to identify the limit in $L^{2}(\ree; \IC^m)$ that we have to take principal values.
\end{rem}

We also provide a result implying that the operators $\mS_t^\Le$ are the unique bounded operators $L^2(\R^n; \IC^m) \to \dot W^{1,2}(\R^n; \IC^m)$ depending strongly continuously on $t \in \R$, for which the representation \eqref{eq:cvgamma1} holds true.

\begin{prop}\label{prop:uniqbl} Let $f\in C_{0}^\infty(\R^{n}; \IC^m)$ with $\int_{\R^n} f=0$ if $n=1,2$. Let $\chi_{\varepsilon}(s)= \frac{1}{\varepsilon} \chi(\frac{s}{\varepsilon})$ with $\varepsilon>0$ and $\chi\in C_{0}^\infty(\R)$ satisfying $\int_{\R}\chi(s)\, ds=1$. Set $f_{\varepsilon}(s,x) := \chi_{\varepsilon}(s)g(x)$. Then $\Le ^{-1}f_{\varepsilon}$ belongs to $C_0(\R; \dot W^{1,2}(\R^n; \IC^m))$  and for all $t\in \R$, $(\Le ^{-1}f_{\varepsilon})_{t}$ converges in $\dot W^{1,2}(\R^n; \IC^m)$ to $\mS_{t}^\Le  f$. 
\end{prop}
\section{The case \texorpdfstring{$p\ne 2$}{p =/= 2}}\label{sec:pnot2}

Let us mention that the case when $p- 2$ is small could be treated similarly to the case $p=2$ without the representation by layer potentials and assuming only G\aa{}rding's inequality \eqref{eq:Ga}. This would use a basic extension of Lemma~\ref{lem:L-1} and Lemma~\ref{lem:NTH}, taking into account that $\Le: \dot W^{1,p}(\ree) \to \dot W^{-1,p}(\ree)$ remains invertible for such $p$ due to \u{S}ne{\u{\ii}}berg's lemma~\cite{Sn}. But this does not apply when $p$ gets ``far'' from $2$.

Henceforth, we assume \eqref{eq:accrassumption} and begin with a lemma analogous to Lemma~\ref{lem:NTH} in our range of $p$.

\begin{lem}\label{lem:NTHp} Let $\wt G \in C_{0}^\infty(\ree;\IC^m)$ be such  that $\wt G=\div_{x}G^\sharp$ for some $G^\sharp \in C_{0}^\infty(\ree; \C^{nm})$. Set $\wt H:=(\Le^*) ^{-1}(\wt G)$,  $G:=\partial_{t}\wt G$ and $H:=\partial_{t}\wt H$.  Let $1<p<\infty$. 

\begin{enumerate}
 \item  If $p\in \cH_{\Le }$, then $\partial_t^k H \in C_{0}(\R; L^{p'}(\R^n; \IC^m))$ for all integers $k \geq 0$ and  $\|\NTone H\|_{p'}<\infty$.
 
 \item If $p'\in \cH_{\Le ^*}$, then $\partial_t^k H \in C_{0}(\R; W^{1,p'}(\R^n; \IC^m))$ for all integers $k \geq 0$ and $\|\NTone H\|_{p'} < \infty$ as well as $\|\NTone(\nabla H)\|_{p'}<\infty$.
\end{enumerate}
In both statements, the distinguished variable in the regularity estimates is $t \in \R$.
\end{lem}

\begin{proof}  Let $a,b\in \R$ such that $\supp \wt G \subset [a,b] \times \R^n$. We may assume $b\ge 2$ for simplicity. Note that we have the assumptions of Proposition \ref{prop:representation} for $\wt G, G$ and so the representations apply. 

We look at (i) first. By \eqref{eq:cvgamma3} we have 
\begin{align}
\label{eq1:representation}
 H_{t}=\partial_{t}\wt H_{t}= \int_{\R} \mS_{t-s}^{\Le ^*} \pd_{s} \wt G_{s}\ ds= \int_{\R} \mS_{t-s}^{\Le ^*}  G_{s}\ ds
\end{align}
with $G_{s}=\div_{x} (\pd_{s} G^\sharp_{s})\in \dot W^{-1,p'}(\R^n; \IC^m)$ by assumption. Due to Lemma \ref{lem:bl}(ii) --- but replacing $(p,\Le )$ by $(p',\Le ^*)$ therein --- we can bound the norm of $\mS_{t}^{\Le ^*}$ as a bounded operator from $\dot W^{-1,p'}(\R^n; \IC^m)$ to $L^{p'}(\R^n; \IC^m)$ uniformly in $t \in \R \setminus \{0\}$. From Minkowski's inequality and the fact that $G$ is smooth with compact support, we can infer
\begin{align*}
 \|H_t\|_{p'} \lesssim \int_{\R} \|\mS_{t-s}^{\Le ^*} G_{s}\|_{{p'}} \ ds \lesssim \int_a^b \|G_s\|_{\dot W^{-1,p'}} \ ds < \infty
\end{align*}
uniformly for all $t \in \R$. Owing to Lemma~\ref{lem:L-1}, an analogous formula applies to $\partial_{t}^kH$ with $\partial_{s}^kG$ in the integral for $k$ an integer and so we also have $\sup_{t\in \R} \|\partial_{t}^kH_{t}\|_{{p'}}<\infty$. Continuity of $t \mapsto \partial_t^k H_t$ and the limits at $\pm\infty$ both in $L^{p'}(\R^n; \IC^m)$ follow by applying the dominated convergence theorem to \eqref{eq1:representation}: Indeed, for $s \in \R$ fixed, Lemma~\ref{lem:bl}(ii) shows that $t \mapsto \mS_{t-s}^{L^*} \partial_s^k G_s$ is continuous on $\R \setminus \{s\}$ and bounded with values in $L^{p'}(\R^n; \IC^m)$. 
  
Finally, we prove to the maximal estimate $\|\NTone H\|_{p'}<\infty$. Proceeding as when $p=2$ in the proof of Lemma~\ref{lem:NTH} --- and with the same notation --- it suffices to show that $F(x)=  \int_{0}^{8b} |\partial_{t}H(t,x)|\ dt$ and $F_{1}(x)= \int_{2b}^{\infty} |\partial_{t}H(t,x)|\ dt$ belong to $L^{p'}(\R^n)$. First, 
$$
     \int_{\R^n} |F(x)|^{p'}\ dx \le (8b)^{p'/p}\int_{\R^n} \int_{0}^{8b} |\partial_{s}H(s,x)|^{p'}\ ds \, dx \le (8b)^{p'}\sup_{t\in \R} \|\partial_{t}H_{t}\|^{p'}_{{p'}}<\infty.
$$
For $F_{1}$ we take a different approach. By Minkowski's inequality
$$
    \|F_{1}\|_{{p'}} \le \int_{2b}^\infty \|\partial_{t}H_{t}\|_{{p'}}\ dt
$$
and we can use integration by parts twice along with $G_{s}=\partial_{s}\wt G_{s}$ to obtain from the representation \eqref{eq:cvgamma2},
$$
\partial_{t}H_{t}=  \int_{a}^b \mS_{t-s}^{\Le ^*} \partial_{s}G_{s}\ ds= \int_{a}^b \partial_{s}^2\mS_{t-s}^{\Le ^*} \wt G_{s}\ ds.
$$
Note that because of $t -b > b$ we stay away from $t-s=0$ and can use the decay of the single layer.  More precisely, the norm of $\partial_{s}^2\mS_{t-s}^{\Le ^*}: \dot W^{-1,p'} \to L^{p'}$ is bounded by $|t-s|^{-2}$ due to Lemma~\ref{lem:bl} and we obtain $\|\partial_{t}H_{t}\|_{{p'}} \lesssim t^{-2}$, which in turn warrants $\|F_{1}\|_{{p'}}<\infty$. This completes the proof of (i).

We turn to (ii). Using \eqref{eq:cvgamma1} and \eqref{eq:cvgamma3} , we have $\nabla_{x} \partial_t^k H(t,x)= \mathrm{p.v.} \!\int_{\R} \nabla_{x}\mS_{t-s}^{\Le ^*} \partial_s^k G(s,x) \ ds$ in $L^2(\ree; \IC^{mn})$ and for fixed $t$ the integrals are \emph{bona fide} Bochner integrals in $L^2(\R^n; \IC^{nm})$. Again by \eqref{eq:cvgamma3}, we have $\pd_{t}\partial_t^k H_{t} = \int_{\R} \partial_t \mS_{t-s}^{\Le ^*} \partial_s^k G_{s} \ ds$ with the same meaning. From here, the proof of $t$-regularity is entirely analogous to (i), relying instead on Lemma~\ref{lem:bl}(i) but with $(p',\Le ^*)$ replacing $(p,\Le )$ as our assumption is $p'\in \cH_{\Le ^*}$. In particular, we obtain $\|\NTone H\|_{p'}<\infty$ by the same argument.

As for the non-tangential maximal estimate $\|\NTone(\nabla H)\|_{p'}<\infty$, we follow again the proof for the $p=2$ case. We also have to estimate $\wt F(x)=  \int_{0}^{8b} |\partial_{t}\nabla_{x}H(t,x)|\ dt $ and $\wt F_{1}(x)= \int_{2b}^{\infty} |\partial_{t}\nabla_{x} H(t,x)|\ dt$ in $L^{p'}(\R^n)$. First,
$$
     \int_{\R^n} |\wt F(x)|^{p'}\ dx \le (8b)^{p'/p}\int_{\R^n} \int_{0}^{8b} |\partial_{t}\nabla_{x}H(t,x)|^{p'}\ dt \, dx \le (8b)^{p'} \sup_{t\in \R} \|\nabla_{x}(\partial_{t}H)\|_{{p'}}^{p'}<\infty.
$$
Next,
$$
     \|\wt F_{1}\|_{{p'}} \le \int_{2b}^\infty \|\partial_{t}\nabla_{x}H_{t}\|_{{p'}}\ dt.
$$
For $t$ in this range we have $t-s>b$ in the integral that represents $\partial_{t}\nabla_{x}H_{t}$ and integrating by parts twice, we obtain
    $$
\partial_{t}\nabla_{x}H_{t}=  \int_{a}^b \nabla_{x}\mS_{t-s}^{\Le ^*} \partial_{s}G_{s}\ ds= \int_{a}^b \partial_{s}^2\nabla_{x}\mS_{t-s}^{\Le ^*} \wt G_{s}\ ds.
$$
Since the $L^{p'}$-operator norm of $\partial_{s}^2\nabla_{x}\mS_{t-s}^{\Le ^*}$ is controlled by $|t-s|^{-2}$, see Lemma~\ref{lem:bl}, 
we have  $ \|\partial_{t}\nabla_{x}H_{t}\|_{{p'}} \lesssim t^{-2}$, which warrants $ \|\wt F_{1}\|_{{p'}}<\infty.$   
\end{proof}

\begin{rem} 
We have not tried to get optimal hypotheses on $G$ for obtaining the desired estimates. In (ii) we did not use $\wt G=\div_{x}G^\sharp$. For (i), it can also be lifted provided we have $C_{0}^\infty(\R^n) \subset \dot W^{-1,p'}(\R^n) $, which holds for $p'>\frac n{n-1}$ by Sobolev embeddings. It is simpler and enough for us, however, to make this assumption throughout.
\end{rem}

\subsection{Proof of Theorem \ref{thm:uniqueDir} when \texorpdfstring{$p\ne 2$}{p =/= 2}}

We assume $p'\in \cH_{\Le ^*}$. We consider a weak solution to $\Le  u=0$ on $\reu$ such that $\NT u \in L^{p}(\R^n)$ and
\begin{align*}
 \lim_{t \to 0} \bariint_{W(t,x)} |u(s,y)|\ {ds \, dy} = 0
\end{align*}
for a.e.\  $x\in \R^n$. Our task is to show $u=0$ almost everywhere.
   
This time, we pick $G$ of the form $G=\partial_{t}\wt G$ with $\wt G \in C_{0}^\infty(\ree; \IC^m) $ and $\wt G=\div_{x} G^\sharp$. Assume, we had managed to prove $\dual u G=0$, that is, $\dual {\partial_{t}u} {\wt G}=0$. Then $\dual {\nabla_{x} \partial_{t}u} { G^\sharp}=0$, where $G^\sharp$ is an arbitrary test function in $\reu$. Hence, $\partial_t u \in L_{loc}^2(\reu; \IC^m)$ is independent of $x$ and we obtain $u(t,x)=g(t)+f(x)$ with $f \in L_{loc}^1(\R^n; \IC^m)$ and $g: (0, \infty) \to \IC^m$ continuous. Let
\begin{align*}
 v(t,x):=\bariint_{W(t,x)} u(s,y)\ ds \, dy = \barint_{t/2}^{2t} g(s)\ ds + \barint_{B(x,t)} f(y)\ dy .
\end{align*}
We know $v(t,x)\to 0$ as $t \to 0$ for a.e.\ $x \in \R^n$. Applying Lebesgue's differentiation theorem to $f$, it follows that $\barint_{t/2}^{2t} g(s)\ ds$ has a limit when $t\to 0$. Call it $\alpha \in \IC^m$. Then $\alpha + f(x)=0$ almost everywhere, which in turn implies that $u$ is independent of $x$. But due to $\NT u \in L^p(\R^n)$ we must have $\NT u = 0$. This  yields $u =  0$ as desired.

Next, the proof of $\dual u G=0$ is  line by line the same as the one for $p=2$ in Section~\ref{Thm1 p=2}. Indeed, thanks to Lemma~\ref{lem:NTHp}, $h:=H(0,\cdot)\in \dot W^{1,p'}(\R^n;\IC^m)$, where $H:=(\Le ^*)^{-1}G$. Thus, existence for $(R)_{p'}^{\Le ^*}$ yields a solution  to $\Le ^*H_{1}=0$ in $\reu$ with  $\NT (\nabla H_{1}) \in L^{p'}(\R^n)$ and boundary trace $h$. The non-tangential estimates needed to run the argument are $\NT u\in L^{p}(\R^n)$ by assumption, $\NTone(\nabla(H-H_{1}))\in L^{p'}(\R^n)$ by Lemma \ref{lem:NTHp}(ii) and the construction of $H_{1}$, and $\NTone((H-H_{1})/t)\in L^{p'}(\R^n)$ by  Lemma \ref{lem:trace}.

\begin{rem}\label{rem:modifp} We can modify the proof of convergence to $0$ of $J_{\varepsilon}$ as in Remark \ref{rem:modif} if we only assume $\barint_{t/2}^{2t}|u(s,\cdot)|\, ds \to 0$ in $ L^p_{loc}(\R^n)$ as $t\to 0$.  \end{rem}

\subsection{Proof of Theorem \ref{thm:uniquereg} when \texorpdfstring{$p\ne 2$}{p =/= 2}}

Let $p\in \cH_{\Le }$. We assume that $\Le u=0$ on $\reu$, that $\NT (\nabla u) \in L^p(\R^n)$ and that 
\begin{align*}
 \lim_{t \to 0} \bariint_{W(t,x)} |u(s,y)|\ {ds \, dy} = 0
\end{align*}
for a.e.\  $x\in \R^n$. We have to show $u=0$ almost everywhere. To this end, we pick $G$ as in the previous proof. The same argument then shows that it suffices to check $\dual u G=0$, noting that  we can use $\NTone(u/t) \in L^p(\R^n)$ from Lemma \ref{lem:trace} to deduce $u = 0$ once it has been seen to depend on $t$ only.
  
Next, the argument to show $\dual u G=0$ is identical to the one for $p=2$ presented in Section~\ref{Thm2 p=2}:  Existence for $(D)_{p'}^{\Le ^*}$ yields a solution  to $\Le ^*H_{1}=0$ in $\reu$ with  $\NT (H_{1}) \in L^{p'}(\R^n)$ and boundary trace $h:=H(0,\cdot)\in L^{p'}(\R^n;\IC^m)$, where $H:=(\Le ^*)^{-1}G$. The required non-tangential estimates are $\NTone (u/t)\in L^{p}(\R^n)$ as seen above as well as $\NTone(H-H_{1})\in L^{p'}(\R^n)$ by Lemma~\ref{lem:NTHp}(i) and the construction of $H_{1}$.

\subsection{Proof of Corollary \ref{cor:wp}} 
\label{sec:proofSF}

Assume $1<p<\infty$ with $p'\in \cH_{\Le ^*}$.
Theorem 1.6 of \cite{AM} shows that well-posedness of $(R)_{p'}^{\Le ^*} $ is equivalent to well-posedness of a modified Dirichlet problem $ (\wt D)_{p}^\Le $.
Combining (1) and (2) of Theorem~1.9 in \cite{AM}, we have that the compatible well-posedness of  $(R)_{p'}^{\Le ^*} $ is equivalent to the compatible well-posedness of this modified Dirichlet problem $ (\wt D)_{p}^\Le $,
\begin{equation*}
 (\wt D)_{p}^\Le   \qquad\qquad
\begin{cases}
  \Le u=0   & \text{on}\ \reu, \\
   S(t\nabla u)\in L^{p}(\R^n)   \\
   \lim_{t \to 0} u(t,\cdot) = f & \text{in  } L^p(\R^n;\IC^m).\end{cases}
\end{equation*}
The following was shown in \cite{AM}, Corollary 1.4:  Any solution to $\Le u=0$ in $\reu$ such that $S(t\nabla u)\in L^{p}(\R^n)$ for $p$ in this range, is, up to a constant, in $C_{0}([0,\infty); L^p(\R^n;\IC^m))$, yielding the existence of $u(0,\cdot)$ and so the limit makes sense. Moreover, there are estimates
$$ \sup_{t\ge 0} \|u(t, \cdot )\|_{p} \lesssim \|S(t\nabla u)\|_{p}$$
and
$$ \|\NT(u)\|_{p} \lesssim  \|S (t\nabla u)\|_{p}, $$
as well as an almost everywhere limit
$$
\lim_{t \to 0} \bariint_{W(t,x)} |u(s,y)-u(0,x)|^2\ {ds \, dy} = 0.
$$
(In fact, a weaker form is stated in \cite{AM} but this stronger form holds and is a consequence of \cite{AusSta}, Theorem 9.9.) 

Thus, given $f\in L^p(\R^n;\IC^m)$, the unique solution (the constant is eliminated)  of $(\wt D)_{p}^\Le $  with data $f$ satisfies
$\|S(t\nabla u)\|_{p}\lesssim \|f\|_{p}$. 
Note then that  $\|f\|_{p}=\|u(0,\cdot)\|_{p} \le \sup_{t\ge 0} \|u(t, \cdot )\|_{p}$ and also $\|u(0,\cdot)\|_{p}\lesssim 
\|\NT u \|_{p}$ by  Fatou's Lemma. In particular, $u$ is a solution to $(D)_{p}^\Le $  and \eqref{eq:equiv}, \eqref{eq:cvL2} hold. Uniqueness of $(D)_{p}^\Le $ on the other hand follows from Theorem~\ref{thm:uniqueDir}. 
\section{Regularity  with Hardy Sobolev data vs Dirichlet with BMO or H\"older continuous data.}
\label{sec:other}

This section is mostly contained in \cite{AM} but we feel it is informative to review it in light of what we just proved. 
When $\frac{n}{n+1}<p\le 1$, the regularity problem becomes the following: given $f\in \dot H^{1,p}(\R^n; \IC^m)$,  solve uniquely (modulo constants),
\begin{equation*}
(R)_{p}^\Le  \qquad\qquad
\begin{cases}
   \Le u=0   & \text{on}\ \reu, \\
    \NT(\nabla u)\in L^{p}(\R^n)    \\
    \lim_{t \to 0} \bariint_{W(t,x)} |u(s,y)-f(x)|\ {ds \, dy}=0 &  \text{for a.e. } x\in \R^n.\end{cases}
\end{equation*}
Here, $\dot H^{1,p}(\R^n; \IC^m)$ is the Hardy-Sobolev space of order 1 above the real Hardy space $H^p$.   
The dual problem is the following Dirichlet problem: given $f\in \dot \Lambda^\alpha(\R^n; \IC^m)$, $0\le \alpha<1$,  solve uniquely (modulo constants)
\begin{equation*}
(D)_{	\alpha}^{\Le^*}  \qquad\qquad
\begin{cases}
  \Le^* u=0   & \text{on}\ \reu, \\
  C_{\alpha}(t \nabla u) \in L^\infty(\R^n)    \\
     \lim_{t \to 0} u(t,\cdot) = f & \text{in } \mathcal{D}'(\R^n;\IC^m)/\IC^m, \\
     \lim_{t \to \infty} u(t,\cdot) = 0 & \text{in } \mathcal{D}'(\R^n;\IC^m)/\IC^m,\end{cases}
\end{equation*}
where 
$$
C_{\alpha}F(x):= \sup \bigg(\frac {1}{r^{2\alpha}  |B(y,r)|}\iint_{T_{y,r}}   |F(t,z)|^2\ \frac{dtdz}{t}\bigg)^{1/2},
$$
taken over all open balls $B(y,r)$ containing $x$, with $T_{y,r}=(0,r)\times B(y,r)$. Here, $\dot \Lambda^\alpha(\R^n; \IC^m)$ designates $ BMO(\R^n;\IC^m)$ for $\alpha=0$  and for $\alpha>0$ it designates the homogeneous H\"older space of exponent $\alpha$.  Note the condition at infinity that does not follow from the interior control in general.  As the interior condition is on the gradient, the problem is posed modulo constants.  

The situation of interest is when the interval of Hardy spaces coincidence contains exponents below $p=1$, as is the case under (DGNM).

\begin{thm}\label{thm:unique} Let  $\frac{n}{n+1}<p\le 1$ with $p\in \cH_{\Le}$ and let $\alpha=n(\frac{1}{p}-1)$. 
\begin{enumerate} 
 \item Existence for   $(R)_{p}^{\Le} $   implies  uniqueness for $(D)_{\alpha}^{\Le^*}$.
 \item Existence for   $(D)_{\alpha}^{\Le^*} $   implies  uniqueness for $(R)_{p}^{\Le}$.
 \item (Compatible) well-posedness for   $(R)_{p}^{\Le} $   implies  (compatible) well-posedness for $(D)_{\alpha}^{\Le^*}$.
\end{enumerate}
\end{thm}

\begin{proof}  

For item (iii),  the implication concerning well-posedness  is in Theorem 1.6 in \cite{AM}; concerning compatible well-posedness, the implication is contained in Theorem 1.9(1) in \cite{AM}. For (i) and (ii) we recall that   Theorems 1.1, 1.3 and 1.7 of \cite{AM} yield the following properties. 
\begin{itemize}
 \item[(A)] For any weak solution $u$ to $\Le u=0$ on $\reu$ with $\NT(\nabla u)\in L^p(\R^n)$, the conormal derivative $\partial_{\nu_{A}}u|_{t=0}$ exists in $H^p(\R^n;\IC^m)$, $u|_{t=0}$ exists in $\dot H^{1,p}(\R^n; \IC^m)$, and   $u$  is constant if and only if  $u|_{t=0}=0$ and $\partial_{\nu_{A}}u|_{t=0}
=0$ in the respective spaces.  

 \item[(B)]  For any weak solution $w$ to $\Le^*w=0$ on $\reu$ with  $C_{\alpha}(t \nabla w) \in L^\infty(\R^n)$  and $w(t,\cdot)$ converging to $0$ in $\mD'(\R^n; \IC^m)$ modulo constants as $t\to \infty$,  $w|_{t=0}$ exists in $\dot \Lambda^\alpha(\R^n; \IC^m)$ and $\partial_{\nu_{A^*}}w|_{t=0}$ exists in $\dot \Lambda^{\alpha-1}(\R^n; \IC^m)$, and  $w$  is constant if and only if  $w|_{t=0}=0$ and $\partial_{\nu_{A^*}}w|_{t=0}
=0$ in the respective spaces. Here, a distribution is in  $\dot \Lambda^{\alpha-1}(\R^n; \IC^m)$ if it is the divergence of an element in $\dot \Lambda^{\alpha}(\R^n; \IC^{mn})$.  

\item[(C)] With $u$ and $w$ as above, there is a Green's formula
\begin{equation*} 
\label{eq:green}
\dual {\partial_{\nu_{A}}u|_{t=0}} {w|_{t=0}} = \dual {u|_{t=0}}{\partial_{\nu_{A^*}}w|_{t=0}} .
\end{equation*}
Here, the first pairing is the $\dual {H^p(\R^n;\IC^m)} {\dot \Lambda^\alpha(\R^n; \IC^m)} $ sesquilinear duality while the second one is the 
$\dual{\dot H^{1,p}(\R^n; \IC^m)}{\dot \Lambda^{\alpha-1}(\R^n; \IC^m)}$ sesquilinear duality.
\end{itemize}
With this at hand the arguments are very simple. For (i), assume that $w$ is given with $w|_{t=0}=0$. Let $\varphi \in C_{0}^\infty(\R^n; \IC^m)$ and solve $(R)_{p}^\Le$ with $u|_{t=0}=\varphi$ (modulo constants). It follows from (C) that $\dual {\varphi}{\partial_{\nu_{A^*}}w|_{t=0}}=0$.  This means that also $\partial_{\nu_{A^*}}w|_{t=0}=0$ as a distribution, hence $w$ is constant using (B). For (ii), assume that $u$ is given with $u|_{t=0}=0$. Let $\varphi \in C_{0}^\infty(\R^n; \IC^m)$ and solve $(D)_{\alpha}^{\Le^*}$ with $w|_{t=0}=\varphi$ (modulo constants). By (C) we have $\dual {\partial_{\nu_{A}}u|_{t=0}}{\varphi}=0$.  This means $\partial_{\nu_{A}}u|_{t=0}=0$ as a distribution, hence $u$ is constant using (A).  
\end{proof}
\section{Proof of the layer potential representation}\label{sec:prooflayer}

Now, we come to justify \eqref{eq:L-1} of the introduction in proving Proposition \ref{prop:representation}. For this we set for $0<\varepsilon<R<\infty$,  
\begin{equation*}
(\Gamma_{\varepsilon,R}f)_{t}= \int_{\varepsilon<|t-s|<R} \mS_{t-s}^\Le  f_{s} \ ds, \quad t\in \R.
\end{equation*} 
We  give the main properties of these approximants and then show how they converge to $\Le ^{-1}$. The proofs depend only on the case $p=2$ of Lemma \ref{lem:bl}.

\begin{lem}\label{lem:approx} Assume $f\in C_{0}^\infty(\ree;\IC^m)$. If $n=1,2$ assume furthermore that $f=\div_{x}F$ coordinatewise for some $F\in C_{0}^\infty(\R^n; \C^{mn})$. Then for fixed  $0<\varepsilon<R<\infty$, we have $\Gamma_{\varepsilon,R}f \in W^{1,2}(\ree; \IC^m)$ and $t\mapsto (\Gamma_{\varepsilon,R}f)_{t} \in C^\infty(\R; W^{1,2}(\R^n; \IC^m))$ with all derivatives bounded.
\end{lem}

\begin{proof}
Sobolev embeddings yield $f \in W^{-1,2}(\ree; \IC^m)$, except when $n=1,2$, where the condition $f=\div_{x}F$ is also required. This is the case under our assumption. Thus $u=\Le ^{-1}f$ is well-defined in $\dot W^{1,2}(\ree;\IC^m)$. 
Also, $t\mapsto f_{t}$ is in $C_{0}^\infty(\R; L^2(\R^n;\IC^m)\cap \dot W^{-1,2}(\R^n;\IC^m))$, using again that for $n=1,2$ we have $\int_{\R^n} f_{t}=0$ for all $t\in \R$ by assumption. Lemma \ref{lem:bl} yields that $t\mapsto \mS_{t}^\Le$ is uniformly bounded and in particular belongs to $L_{loc}^1(\R; \cL(X,Y))$, where 
\begin{align*}
 (X,Y)= (\dot W^{-1,2}(\R^n;\IC^m), L^2(\R^n;\IC^m)) \qquad \text{or} \qquad (X,Y)= (L^2(\R^n;\IC^m), \dot W^{1,2}(\R^n;\IC^m)).
\end{align*}
By (operator-valued) convolution $t\mapsto (\Gamma_{\varepsilon,R}f)_{t}$ is in $L^\infty(\R; W^{1,2}(\R^n;\IC^m)) \cap L^2(\R; W^{1,2}(\R^n;\IC^m))$.
In particular, $\Gamma_{\varepsilon,R}f\in L^2(\ree;\IC^m)$ and $\nabla_{x}\Gamma_{\varepsilon,R}f\in L^2(\ree; \IC^{nm})$. We next compute $\partial_{t}\Gamma_{\varepsilon,R}f$ by differentiating the integral and  integrating by parts. Setting $w_{t,\varepsilon} :=\mS_{\varepsilon}^\Le  f_{t-\varepsilon}-\mS_{-\varepsilon}^\Le  f_{t+\varepsilon}$, we obtain
\begin{align}
\label{eq:dtst}
\begin{split}
(\partial_{t}\Gamma_{\varepsilon,R}f)_{t} & = \int_{\varepsilon<|t-s|<R} \partial_{t}\mS_{t-s}^\Le  f_{s} \ ds + w_{t,\varepsilon}-w_{t,R} \\
& = -   \int_{\varepsilon<|t-s|<R} \partial_{s}\mS_{t-s}^\Le  f_{s}\ ds +w_{t,\varepsilon}-w_{t,R}   \\
&  =   \int_{\varepsilon<|t-s|<R} \mS_{t-s}^\Le  \partial_{s}f_{s}\ ds \\
& = \Gamma_{\varepsilon,R}(\partial_{s}f)_{t}.
\end{split}
\end{align}
Since $\partial_{s}f$ has the same properties as $f$, we can apply the above to $\Gamma_{\varepsilon,R}(\partial_{s}f)$ and conclude that  $\partial_{t}\Gamma_{\varepsilon,R}f$ is in  $L^\infty(\R; W^{1,2}(\R^n;\IC^m)) \cap L^2(\R; W^{1,2}(\R^n;\IC^m))$. Altogether, we have seen $\Gamma_{\varepsilon,R}f \in W^{1,2}(\ree;\IC^m)$ and that $t\mapsto (\Gamma_{\varepsilon,R}f)_{t}$  is bounded and Lipschitz continuous into $W^{1,2}(\R^n;\IC^m)$. Iterating $t$-derivatives from $\partial_{t}(\Gamma_{\varepsilon,R}f)=\Gamma_{\varepsilon,R}(\partial_{s}f)$ and using Lemma~\ref{lem:bl} yields the claim.
\end{proof}

The following lemma will be useful in the proof of convergence.

\begin{lem}
\label{lem:ConvergenceLemma}
Let $X, Y$ be Banach spaces, $T: [0,\infty) \to \Lop(X,Y)$ uniformly bounded and strongly continuous and let $f \in C_0(\R; X)$. Then for any $s_0 \in [0,\infty]$ there is convergence in $Y$ of
\begin{align*}
 \lim_{s \to s_0} T_s f_{t+s} = T_{s_0} f_{t+s_0},
\end{align*}
uniformly in $t \in \R$. If in addition $f \in L^p(\R; X)$ for some $p \in [1, \infty)$, then convergence also holds in $L^p(\R, dt; Y)$.
\end{lem}

\begin{proof}
We set $f(\pm \infty) := 0$, so that $f$ becomes (uniformly) continuous on $\R \cup \{\pm \infty\}$, viewed as a compact topological space. Since $T$ is uniformly bounded,
\begin{align}
\label{eq1:ConvergenceLemma}
 \|T_s f_{t+s} - T_{s_0}f_{t+s_0}\|_Y
\lesssim \|f_{t+s} - f_{t+s_0}\|_Y + \|(T_{s} - T_{s_0})f_{t+s_0}\|_Y.
\end{align}
In the limit $s \to s_0$ the first term on the right tends to $0$ uniformly in $t \in \R$ since $f$ is uniformly continuous. For the second term we first note that $K:= \{f_{t+s_0} : t \in \R \cup \{\pm \infty\}\}$ is the continuous image of a compact set, and hence is compact in $X$. Thus, the strong convergence $T_s - T_{s_0} \to 0$ as $s \to s_0$ improves to uniform strong convergence on $K$. Therefore we also get convergence to $0$ uniformly in $t \in \R$ for the second term above and the proof of the first claim is complete. 

Now suppose in addition $f \in L^p(\R; X)$. Taking \eqref{eq1:ConvergenceLemma} to the $p$-th power and integrating in $t$ gives
\begin{align*}
 \int_{\R} \|T_s f_{t+s}& - T_{s_0}f_{t+s_0}\|_Y^p \ dt \\
 &\lesssim \int_{(-R,R)}\|f_{t+s} - f_{t+s_0}\|_X^p \ dt + \int_{\R \setminus (-R,R)} \|f_{t+s} - f_{t+s_0}\|_X^p \ dt + \int_{\R} \|(T_{s} - T_{s_0})f_{t+s_0}\|_Y^p \ dt,
\end{align*}
where $R>0$ is a degree of freedom. Given $\eps > 0$, we can first choose $R$ large enough to guarantee that for all $s \in (s_0-1,s_0+1)$ the middle term is bounded by $\eps$. Then the first integral vanishes in the limit $s \to s_0$ by uniform continuity of $f$ and the third one vanishes by dominated convergence taking into account boundedness and strong continuity of $T$.
\end{proof}

In order to proceed, we eventually have to give the abstract definition of the single layer $\mS_{t}^\Le $. All of the following material and further background can be found in \cite{AusSta} and we report here only on the essentials required to follow the line of reasoning. We identify $\IC^{(1+n)m}$ with $\IC^m \times \IC^{nm}$ and represent vectors as $F = \begin{bmatrix} F_\pe \\ F_\pa \end{bmatrix}$ accordingly. 

Following the calculation on p.68 of \cite{AA1}, there are a constant coefficient first order differential operator $D$ acting on $\IC^{(1+n)m}$-valued functions, a bounded multiplication operator $B = B(x)$ on $L^2(\R^n; \IC^{(1+n)m})$ and purely algebraic way of rewriting the elliptic system $\Lop u = f$ as
\begin{align*}
 \partial_t F + DB F = \begin{bmatrix} \Lop u \\ 0 \end{bmatrix}, \quad \text{where } F = \nabla_A u: = \begin{bmatrix} (A\nabla u)_\pe \\ \nabla_x u \end{bmatrix}.
\end{align*}
Here, $u \in \W_{loc}^{1,2}(\ree; \IC^m)$ and the autonomous first order equation for its conormal gradient $F = \nabla_A u$ is understood in the sense of distributions. The multiplication operator $B$ is designed in such a way that
\begin{align}
\label{eq:B}
B \begin{bmatrix} (AF)_\pe \\ F_\pa \end{bmatrix} = \begin{bmatrix} F_\pe \\ (A F)_\pa \end{bmatrix}, \qquad F \in \IC^{(1+n)m}.
\end{align}
The operator $DB$ has an $H^\infty$-functional calculus on $\cH = \clos{\ran(D)}$, the closure of the range of $D$ in $L^2(\R^n; \IC^{(1+n)m})$, allowing us to define a bounded operator $\varphi(DB)$ on $\cH$ for any bounded and holomorphic function $\varphi$ on a suitably large double sector around the real axis. Recall that $\cH$ has been defined in connection with the ellipticity condition \eqref{eq:accrassumption} and that it contains $L^2(\R^n; \IC^m) \times \{0\}$. With $\chi^\pm$ the indicator functions of $\IC^\pm$, functional calculus provides a means of defining $e^{\mp tD B}\chi^{\pm}(D B)$ for $t>0$ as a bounded operator on $\cH$. These are the bounded holomorphic $C_0$-semigroups generated by $\mp DB$ on $\chi^\pm(DB) \cH$.

From \cite{AusSta}, p.100-101 with correction of an unfortunate typo in (82) and (84), we then have
\begin{equation} 
\label{eq:st}
\mS_{t}^\Le f:= \begin{cases}
  -\bigg(D^{-1} e^{-tD B}\chi^{+}(D B) \begin{bmatrix} f \\ 0\end{bmatrix} \bigg)_{\no}   & \text{if } t>0, \\
\bigg(D^{-1} e^{-tD B}\chi^{-}(D B) \begin{bmatrix} f \\ 0\end{bmatrix}\bigg)_{\no}      &  \text{if } t<0,
\end{cases}
\end{equation}
 and 
 \begin{equation} 
\label{eq:gradst}
\gradA\mS_{t}^\Le f= \begin{cases}
 +  e^{-tDB}\chi^{+}(D B) \begin{bmatrix} f \\ 0\end{bmatrix}    & \text{if } t>0, \\
 - e^{-tD B}\chi^{-}(D B) \begin{bmatrix} f \\ 0\end{bmatrix}      &  \text{if } t<0.
\end{cases}
\end{equation}
Here,  $D^{-1}$ is  the closed extension of $D^{-1}:{\ran(D)}\to\dom(D) $ to $\overline{\mathsf{R}(D)} = \mathcal{H}$ with values in the abstract $D$-adapted Sobolev space $\dot{\mathcal{H}}_D^1$.

\begin{proof}[Proof of Proposition  \ref{prop:representation}]
We begin with \eqref{eq:cvgamma1}. As $f \in \dot W^{-1,2}(\ree)$, we can define $u:= \Lop^{-1}f$ in $\dot W^{1,2}(\ree)$. Defining the conormal gradient $F:= \nabla_A u$, we have $F_t \in \cH$ for every $t>0$ by construction and $F \in L^2(\ree; \IC^{(1+n)m})$. We have seen above that in the distributional sense $pd_{t}F+ \P\M F= \begin{bmatrix} f \\ 0\end{bmatrix}$ but as both $F$ and $f$ are smooth functions of $t$ valued in $L^2(\R^n; \IC^{(1+n)m})$, see Lemma~\ref{lem:L-1}, this first order equation also holds in the strong sense. Replacing $\begin{bmatrix} f \\ 0 \end{bmatrix}$ by this differential equation in \eqref{eq:gradst} and integrating by parts, we obtain for all $t\in \R$, $0<\varepsilon<R<\infty$, 
\begin{align*}
 \int_{\varepsilon<|t-s|<R}\gradA\mS_{t-s}^\Le  f_{s}\ ds& =
 \e^{-\varepsilon \P\M} \chi^{+}(\P\M) F_{t-\varepsilon}
+ \e^{\varepsilon \P\M} \chi^{-}(\P\M) F_{t+\varepsilon}
\\
& \qquad  -\e^{-R \P\M} \chi^{+}(\P\M) F_{t- R}
- \e^{R \P\M} \chi^{+}(\P\M) F_{t+R}.
\end{align*}
Limits in $\eps, R$ for the terms on the right all fall under the scope of Lemma~\ref{lem:ConvergenceLemma} applied with $X = Y = \cH \subset L^2(\R^n; \IC^{(1+n)m})$. Thus,
\begin{align}
\label{eq1:representation_proof}
\lim_{\varepsilon\to 0, R\to \infty}\int_{\varepsilon<|t-s|<R}\gradA\mS_{t-s}^\Le f_{s}\ ds= \chi^{+}(\P\M) F_{t} + \chi^{-}(\P\M) F_{t}=F_{t},
\end{align}
in $L^2(\R^n; \IC^{(1+n)m})$ uniformly in $t \in \R$ and in $L^2(\R; L^2(\R^n; \IC^{(1+n)m})) \simeq L^2(\ree; \IC^{(1+n)m})$. Taking the $\pa$-component, we obtain in particular convergence of $\nabla_{x} \Gamma_{\varepsilon,R}f$ to $F_{\ta}=\nabla_{x}u=\nabla_{x}\Le ^{-1}f$ in $L^2(\ree; \IC^{nm})$. As for convergence of $\partial_t \Gamma_{\varepsilon,R}f$, we keep \eqref{eq:B} in mind, multiply the previous equation by the bounded operator $B$ on $L^2(\R^n; \IC^{(1+n)m})$ and take the $\pe$-component, to give
\begin{align*}
\lim_{\varepsilon\to 0, R\to \infty}\int_{\varepsilon<|t-s|<R} (B \gradA\mS_{t-s}^\Le f_{s})_\pe \ ds 
= \lim_{\varepsilon\to 0, R\to \infty}\int_{\varepsilon<|t-s|<R} \partial_t \mS_{t-s}^\Le f_{s} \ ds
= (B F_t)_\pe = \partial_t u,
\end{align*}
in $L^2(\R^n; \IC^{m})$ uniformly in $t \in \R$ and in $L^2(\ree; \IC^{m})$. Now, let us have a look at \eqref{eq:dtst}. We have seen that the integral in the first line enjoys the desired convergence. The convergence of $w_{t,\varepsilon}-w_{t,R}$ will once again be a direct application of Lemma~\ref{lem:ConvergenceLemma}. Indeed, $\mS_{t}^\Le$ viewed as a bounded operator $(L^2 \cap \dot W^{-1,2})(\R^n; \IC^m) \to W^{1,2}(\R^n; \IC^m))$ is uniformly bounded and strongly continuous with respect to $t\ \in \R$ from Lemma~\ref{lem:bl} and $t\mapsto f_{t}$ valued in $(L^2 \cap \dot W^{-1,2})(\R^n; \IC^m)$ is continuous with compact support. Thus,
\begin{align}
\label{eq2:representation_proof}
 \lim_{\eps \to 0, R \to \infty} w_{t,\varepsilon}-w_{t,R} = (\mS_{0}^\Le f_t - \mS_{0}^\Le f_t) - 0 = 0
\end{align}
even in $W^{1,2}(\R^n; \IC^m)$ uniformly in $t \in \R$ and in $L^2(\R;W^{1,2}(\R^n ; \IC^m))$. In particular, all four lines in \eqref{eq:dtst} share convergence in $L^2(\ree; \IC^m)$ to the limit $\partial_t u = \partial_t \Le^{-1}ft$. By looking just at $\partial_t \Gamma_{\varepsilon,R}f$, we complete the proof of \eqref{eq:cvgamma1}. 

However, the other lines of \eqref{eq:dtst} -- together with the equality $\Le ^{-1}(\partial_{t}f)=  \partial_{t} \Le ^{-1}f$ in $L^{2}(\ree; \IC^m)$ noted in Lemma~\ref{lem:L-1}-- also give all the limits stated in \eqref{eq:cvgamma2} and those stated in \eqref{eq:cvgamma3} in the sense of $L^2(\R^n; \IC^m)$-convergence, uniformly in $t \in \R$. The missing uniform $L^2(\R^n; \IC^m)$-convergence for $\int_{\R} \nabla_x \mS_{t-s}^\Le  (\partial_{s}f_{s})\ ds$ follows from \eqref{eq1:representation_proof} with $f$ replaced by $\partial_s f$ and that of $\int_{\R} \nabla_x \partial_{t}\mS_{t-s}^\Le  f_{s}\ ds$ is a consequence of \eqref{eq:dtst} since we have uniform $W^{1,2}(\R^n; \IC^m)$-convergence 
of the error terms in \eqref{eq2:representation_proof}.

Finally, the integrals in \eqref{eq:cvgamma3} are norm convergent in $W^{1,2}(\R^n; \IC^m)$ for fixed $t$ since for all integers $k\geq 0$, Lemma~\ref{lem:bl} guarantees  that $\partial_s^k S_s: (L^2 \cap \dot W^{-1,2})(\R^n; \IC^m) \to W^{1,2}(\R^n; \IC^m)$ is uniformly bounded in $s \in \R$ and by assumption $s \mapsto \partial_s^k f_s$ is continuous with compact support valued in $(L^2 \cap \dot W^{-1,2})(\R^n; \IC^m)$.
\end{proof}

Next, we shall explain how the properties stated in Lemma~\ref{lem:bl} follow from \cite{AusSta}.

\begin{proof}[Proof of Lemma~\ref{lem:bl}] 
This is  Theorem 12.6, items (1), (3) and (5), of \cite{AusSta}, the  case $p=2$ being mostly from \cite{R}, except for the global strong continuity and the limits at $\pm\infty$. Only the strong limits at $0^\pm$ are explained there. But continuity at other points is even easier with the arguments there, using the boundedness properties of the $DB$-semigroups in Corollary~8.3 in \cite{AusSta}. This being said, the limits at $\pm\infty$ come similarly from the fact that a holomorphic $C_{0}$-semigroup converges strongly to $0$ at $\infty$ on the closure of the range of its generator.
\end{proof}

We conclude with:

\begin{proof}[Proof of Proposition~\ref{prop:uniqbl}]
For any $\varepsilon>0$, we have that $f_{\varepsilon}$ belongs to $\dot W^{-1,2}(\R^{n+1}; \IC^m)$, hence $\Le ^{-1}f_{\varepsilon}$ exists in $\dot W^{1,2}(\R^{n+1}; \IC^m)$ and Lemma \ref{lem:L-1} shows that in fact it belongs to $C_0(\R; \dot W^{1,2}(\R^n; \IC^m))$. It remains to prove the convergence. By \eqref{eq:cvgamma1}, we have
\begin{align*}
\nabla_{x}(\Le ^{-1}f_{\varepsilon})(t,x)= \mathrm{p.v.} \!\int_{\R} \nabla_{x}\mS_{t-s}^\Le f(x)  \, \chi_{\varepsilon}(s) \ ds, \qquad \text{in  $L^{2}(\R^{n+1}; \IC^{mn})$}.
\end{align*}
Boundedness and strong continuity of $\mS_t^\Le$ yield that the integral on the right converges in norm and the equality holds in $C_0(\R; L^{2}(\R^n; \IC^{mn}))$ for fixed $\varepsilon>0$. Now we fix $t\in \R$. Changing variables 
\begin{align*}
\nabla_{x}(\Le ^{-1}f_{\varepsilon})(t,x) = \int_{\R} \nabla_{x}\mS_{s}^\Le f(x)  \, \chi_{\varepsilon}(t-s) \ ds
\end{align*}
and using the continuity of $s\mapsto \nabla_{x}\mS_{s}^\Le f$ valued in $L^2(\R^n; \IC^{nm})$, the integral converges to $\nabla_{x}\mS_{t}^\Le f$ in $L^2(\R^n; \IC^{mn})$ as $\varepsilon$ tends to $0$.
\end{proof}
\section{Generic technical lemmas}\label{sec:technical}

For the convenience of the reader, here are some technical lemmas involving non-tangential maximal functions used throughout the paper.  

\begin{lem} \label{lem:infinity} Let $0<q, p<\infty$. Let $F:\reu \to \R$ be a measurable function with $\|\NTq F\|_{p}<\infty$. 
Then  for all $x\in \R^n$,
$$
\lim_{t \to \infty} \bariint_{W(t,x)} |F(s,y)|^q\ {ds \, dy} = 0.
$$
\end{lem}

\begin{proof}
Let $G$ be the $q$-adapted non-tangential function of $F$ with parameters $c_{0}=2, c_{1}=2$. Thus $\|G\|_{p} \sim  \|\NTq F\|_{p}<\infty$.  Next, for all $z\in B(x,t)$,
$$
\bigg(\bariint_{W(t,x)} |F(s,y)|^q\ {ds \, dy}\bigg)^{p/q} \le 2^{np/q}G(z)^p,
$$
and hence
$$
\bigg(\bariint_{W(t,x)} |F(s,y)|^q\ {ds \, dy}\bigg)^{p/q} \lesssim \barint_{B(x,t)} G(z)^p\, dz  \lesssim t^{-n} \|G\|_{p}^p.$$
The conclusion follows. 
\end{proof}

\begin{lem}\label{lem:product} Let $1<p<\infty$. Let $F, H:\reu \to \R$ be measurable functions with $\|\NT F\|_{p}<\infty$  and $\|\NT H\|_{p'}<\infty$. Then for any fixed $0<\varepsilon < R<\infty, $
$$
\iint_{(\varepsilon,R)\times \R^n} |FH|\ dx \, dt \lesssim \|\NT F\|_{p}  \|\NT H\|_{p'}
$$
and 
$$
\lim_{M\to \infty} \iint_{(\varepsilon,R)\times \{|x|>M\}} |FH|\ dx \, dt =0.
$$
\end{lem}

\begin{proof} By covering the interval $(\varepsilon, R)$ with a finite number of intervals of the form $(c_{0}^{-1}a, c_{0}a)$, we may reduce to a single such interval.  In that case, the averaging trick in the $x$ variable and H\"older's inequality show that 
$$
 \int_{{c_0}^{-1}a}^{c_0 a} \int_{\R^n} |FH|\ dx \, dt = \int_{\R^n} a(c_{0}-c_{0}^{-1}) \bigg(\bariint_{W(a,y)} |FH|\ dx \, dt \bigg) dy\le a(c_{0}-c_{0}^{-1}) \|\NT F\|_{p}  \|\NT H\|_{p'} <\infty.
$$
The limit follows by dominated convergence.
\end{proof}

\begin{lem}\label{lem:trace} Let $1\le q <\infty, 1<p<\infty$ and let $H \in W_{loc}^{1,q}(\reu)$ be such that $\|\NTq (\nabla H)\|_{p}<\infty$.  Then  there exists a measurable function $h:\R^n\to \R$ such that for a.e.\ $x\in \R^n$,
\begin{align*}
 \lim_{t\to 0}\bariint_{W(t,x)} |H(s,y)-h(x)|\ {ds \, dy}=0
\end{align*}
as well as
\begin{align*}
 \lim_{t \to 0} \barint_{t/2}^{2t} H(s,\cdot) \ ds = h
\end{align*}
in $L_{loc}^1(\R^n)$.
Moreover, $h\in \dot W^{1,p}(\R^n)$ with $\|\nabla_{x}h\|_{p}\lesssim \|\NTq (\nabla H)\|_{p}$ and
\begin{align*}
 \bigg\|\NTone \bigg(\frac{H-h}{t}\bigg) \bigg\|_{p} \lesssim \|\NTq (\nabla H)\|_{p}.
\end{align*}
 \end{lem}

\begin{proof} It is enough to assume $q=1$ throughout as $\NTone (\nabla H) \le \NTq (\nabla H)$. This is then essentially in \cite{KP}, pp.~461-462,  up to minor modifications of the proof (working directly with averages) and is a simple consequence of Poincar\'e inequalities and change of parameters $c_{0},c_{1}$. Details of this modification are written out for example in Section~6.6 of \cite{AA}.
 \end{proof}

\end{document}